\documentclass[11pt,letter]{article}
\usepackage[pagewise,mathlines]{lineno}
\usepackage{amsfonts,amssymb,enumitem}
\usepackage{amsmath,amsthm}
\usepackage[T1]{fontenc}
\usepackage{lmodern}
\usepackage[utf8]{inputenc}
\usepackage{microtype}
\usepackage{bbm,hyperref}
\hypersetup{
	colorlinks=true,
    linkcolor={red!50!black},
    citecolor={blue!50!black},
    urlcolor={blue!80!black},
	bookmarksopen=true,
	bookmarksnumbered,
	bookmarksopenlevel=2,
	bookmarksdepth=3
}
\usepackage[capitalise]{cleveref}
\usepackage{mathtools}

\usepackage{graphicx}
\usepackage{epsf}
\usepackage[margin=1in]{geometry}
\usepackage{epsfig}
\usepackage{graphics}
\usepackage{color}
\usepackage[normalem]{ulem}
\usepackage[defaultlines=3,all]{nowidow}
\usepackage{todonotes}

\usepackage{tikz}
\usetikzlibrary{positioning,shapes,calc,decorations.pathmorphing,decorations.pathreplacing}

\usepackage{titling}

\newcommand{\Bag}{X}

\usepackage{array}
\newcolumntype{L}[1]{>{\raggedright\let\newline\\\arraybackslash\hspace{0pt}}m{#1}}
\newcolumntype{C}[1]{>{\centering\let\newline\\\arraybackslash\hspace{0pt}}m{#1}}
\newcolumntype{R}[1]{>{\raggedleft\let\newline\\\arraybackslash\hspace{0pt}}m{#1}}
\usepackage{boxedminipage}
\newcommand{\problemdef}[4][XXXEMPTYLABELXXX]{
	\begin{center}
		\begin{boxedminipage}{\textwidth}
			\textsc{{#2}}\ifthenelse{\equal{#1}{XXXEMPTYLABELXXX}}{}{\label{#1}}\\[2pt]
		    \renewcommand{\arrayrulewidth}{0pt}
			\begin{tabular}{@{\hspace{0.007\textwidth}}r@{\hspace{0.007\textwidth}}p{0.87\textwidth}@{\hspace{0.007\textwidth}}}
				\textit{Input:}  & {#3}\\
				\textit{Output:}  & {#4}
			\end{tabular}
		\end{boxedminipage}
	\end{center}
}

\usepackage{etoolbox} %

\newcommand*\linenomathpatch[1]{%
  \cspreto{#1}{\linenomath}%
  \cspreto{#1*}{\linenomath}%
  \csappto{end#1}{\endlinenomath}%
  \csappto{end#1*}{\endlinenomath}%
}
\newcommand*\linenomathpatchAMS[1]{%
  \cspreto{#1}{\linenomathAMS}%
  \cspreto{#1*}{\linenomathAMS}%
  \csappto{end#1}{\endlinenomath}%
  \csappto{end#1*}{\endlinenomath}%
}

\expandafter\ifx\linenomath\linenomathWithnumbers
  \let\linenomathAMS\linenomathWithnumbers
  \patchcmd\linenomathAMS{\advance\postdisplaypenalty\linenopenalty}{}{}{}
\else
  \let\linenomathAMS\linenomathNonumbers
\fi

\linenomathpatch{equation}
\linenomathpatchAMS{gather}
\linenomathpatchAMS{multline}
\linenomathpatchAMS{align}
\linenomathpatchAMS{alignat}
\linenomathpatchAMS{flalign}

\usepackage[sort,compress,numbers]{natbib}
\usepackage{doi}

\makeatletter
\pretocmd{\NAT@citexnum}{\@ifnum{\NAT@ctype>\z@}{\let\NAT@hyper@\relax}{}}{}{}
\makeatother

\newcommand{\tw}{\mathrm{tw}}
\newcommand{\itw}{\mathrm{itw}}
\newcommand{\mms}{\mathrm{mms}}
\newcommand{\tin}{\mathsf{tree} \textnormal{-} \alpha}
\newcommand{\treechi}{\mathsf{tree} \textnormal{-} \chi}

\newcommand{\FF}{\mathcal{F}}
\newcommand{\HH}{\mathcal{H}}
\newcommand{\G}{\mathcal{G}}

\renewcommand{\P}{\textsf{P}}
\newcommand{\NP}{\textsf{NP}}

\newtheorem{theorem}{Theorem}[section]

\newtheorem{corollary}[theorem]{Corollary}
\newtheorem{proposition}[theorem]{Proposition}
\newtheorem{lemma}[theorem]{Lemma}
\newtheorem{observation}[theorem]{Observation}
\newtheorem*{fact}{Fact}

\theoremstyle{definition}
\newtheorem{remark}[theorem]{Remark}

\newtheorem{definition}[theorem]{Definition}
\newtheorem{question}[theorem]{Question}

\newtheorem{numbered-claim}{Claim}
\crefname{numbered-claim}{Claim}{Claims}
\Crefname{numbered-claim}{Claim}{Claims}
\crefname{question}{Question}{Questions}
\Crefname{question}{Question}{Questions}
\crefname{property}{property}{properties}
\Crefname{property}{Property}{Properties}
\crefformat{property}{property~(#2#1#3)}
\Crefformat{property}{Property~(#2#1#3)}
\crefrangeformat{property}{properties~(#3#1#4) to~(#5#2#6)}
\crefmultiformat{property}{properties~(#2#1#3)}{ and~(#2#1#3)}{, (#2#1#3)}{ and~(#2#1#3)}

\title{Treewidth versus clique number. II. Tree-independence number}

\preauthor{}
\author{
\begin{center}
Cl\'{e}ment Dallard\textsuperscript{1}, Martin Milani{\v c}\textsuperscript{2,3}, Kenny \v{S}torgel\textsuperscript{2,4}\\[10pt]
{\small \textsuperscript{1} LIP, ENS de Lyon, France}\\
{\small \textsuperscript{2} FAMNIT, University of Primorska, Koper, Slovenia}\\
{\small \textsuperscript{3} IAM, University of Primorska, Koper, Slovenia}\\
{\small \textsuperscript{4} Faculty of Information Studies in Novo mesto, Slovenia}\\[5pt]
{\footnotesize%
\url{clement.dallard@ens-lyon.fr}\quad
\url{martin.milanic@upr.si}\quad
\url{kennystorgel.research@gmail.com}}
\end{center}
}
\postauthor{}
\predate{}
\date{}
\postdate{}

\begin{document}
\maketitle%

\begin{abstract}
In 2020, we initiated a systematic study of graph classes in which the treewidth can only be large due to the presence of a large clique, which we call $(\tw,\omega)$-bounded.
The family of $(\tw,\omega)$-bounded graph classes provides a unifying framework for a variety of very different families of graph classes, including graph classes of bounded treewidth, graph classes of bounded independence number, intersection graphs of connected subgraphs of graphs with bounded treewidth, and graphs in which all minimal separators are of bounded size.
While Chaplick and Zeman showed in 2017 that $(\tw,\omega)$-bounded graph classes enjoy some good algorithmic properties related to clique and coloring problems, it is an interesting open problem to which extent \hbox{$(\tw,\omega)$-boundedness} has useful algorithmic implications for problems related to independent sets.
We provide a partial answer to this question by identifying a sufficient condition for $(\tw,\omega)$-bounded graph classes to admit a polynomial-time algorithm for the Maximum Weight Independent Packing problem and, as a consequence, for the weighted variants of the Independent Set and Induced Matching problems.

Our approach is based on a new min-max graph parameter related to tree decompositions.
We define the \emph{independence number} of a tree decomposition $\mathcal{T}$ of a graph as the maximum independence number over all subgraphs of $G$ induced by some bag of $\mathcal{T}$.
The \emph{tree-independence number} of a graph $G$ is then defined as the minimum independence number over all tree decompositions of $G$.
Boundedness of the tree-independence number is a refinement of \hbox{$(\tw,\omega)$-boundedness} that is still general enough to hold for all the aforementioned families of graph classes.
Generalizing a result on chordal graphs due to Cameron and Hell from 2006, we show that if a graph is given together with a tree decomposition with bounded independence number, then the Maximum Weight Independent Packing problem can be solved in polynomial time.
Applications of our general algorithmic result to specific graph classes are given in the third paper of the series [Dallard, Milani\v{c}, and \v{S}torgel, Treewidth versus clique number. {III}. Tree-independence number of graphs with a forbidden structure, 2022].
\end{abstract}

\clearpage
\tableofcontents
\clearpage

\section{Introduction}

\subsection{Width parameters of graphs}

One of the approaches for dealing with an algorithmically hard problem is to identify restrictions on the input instances under which the problem becomes solvable in polynomial time.
In the case of graph problems, various measures of structural complexity of graphs, commonly known as \emph{graph width parameters}, have become one of the central tools for this task.
When a width parameter of the input graph is bounded by a fixed constant, this often leads to the development of efficient dynamic programming algorithms, typically applied along some kind of tree-like decomposition of the graph into simpler parts, which is well-behaved according to the specifics of the parameter
(see, e.g.,~\cite{MR985145,MR1105479,MR1042649,MR1739644,MR4020556,MR3126918,MR2857670,MR4288865,MR4402362,MR4189412,MR3967291}).
For a concrete example, consider the \emph{treewidth} of a graph $G$.
This parameter measures, roughly speaking, how similar the graph is to a tree.
A formal definition is given using the concept of a \emph{tree decomposition} of $G$, which is a collection of bags (subsets of the vertex set of $G$) arranged into a tree structure and satisfying certain conditions~\cite{MR742386}.
The treewidth of a graph is defined as the smallest possible size, over all tree decompositions of the graph, of a largest bag in the decomposition minus one.
The key algorithmic result in the context of treewidth is the celebrated metatheorem due to Courcelle~\cite{MR1042649}, which states that any decision problem expressible in monadic second order logic, with quantifiers over vertex and edge subsets,
can be solved in linear time if the input graph is equipped with a tree decomposition of bounded width.
A similar theorem for optimization problems was developed by Arnborg, Lagergren, and Seese~\cite{MR1105479}.

Graphs of bounded treewidth are necessarily sparse.
In 2000, Courcelle and Olariu introduced a more general parameter called \emph{clique-width}~\cite{MR1743732}, which is bounded whenever the treewidth is bounded (see also~\cite{MR2148860}) but whose boundedness is not restricted to sparse graph families.
An algorithmic metatheorem for graph classes with bounded clique-width was developed by Courcelle, Makowsky, and Rotics~\cite{MR1739644}.
Similarly to the algorithmic metatheorems for treewidth, this theorem also applies to decision or optimization problems expressible in monadic second order logic; however, it only allows for quantifiers over vertex subsets.
More recently, several other width parameters capable of handling dense graphs were also introduced and studied, including rank-width~\cite{MR2232389}, Boolean-width~\cite{MR2857670,MR3126918}, mim-width~\cite{vatshelle2012new,MR4020556,MR4189412,MR3721445}, and twin-width~\cite{MR4402362,MR4288865,MR4262551}.
Bounded clique-width is equivalent to bounded rank-width or bounded Boolean-width, but implies bounded mim-width as well as bounded twin-width.
Each of these width parameters has some useful algorithmic features.
To mention one more example, the so-called \emph{locally checkable vertex subset} and \emph{vertex partitioning} problems are solvable in polynomial time on any class of bounded mim-width, provided the input graph is given with a branch decomposition of constant width~\cite{MR3126917,MR3126918}.
This includes the \textsc{Independent Set} problem, the \textsc{Induced Matching} problem, and various variants of the \textsc{Dominating Set} problem, as well as the \textsc{$k$-Coloring} and the \textsc{$H$-Homomorphism} problems.

\subsection{Chordal graphs}Despite a growing variety of graph width parameters and associated algorithmic results, there exist some well-structured graph classes containing dense graphs in which all of the above width parameters remain unbounded.
In such a case, different algorithmic approaches are needed.
This is the case, for example, for the well-known class of \emph{chordal graphs} (also known as \emph{triangulated graphs}), defined as graphs in which every cycle of length at least four has a chord.
Chordal graphs correspond to sparse symmetric positive-definite systems of linear equations that admit a particularly efficient Gaussian elimination (see~\cite{MR408312}) and have applications in many other areas, including probabilistic graphical models and continuous optimization (see, e.g., the survey~\cite{DBLP:journals/ftopt/VandenbergheA15}).
They have been extensively studied in the literature from both structural and algorithmic points of view (see, e.g.,~\cite{MR1686154,MR1971502,MR1320296,MR2063679}).
However, the class of chordal graphs has unbounded mim-width, as follows from the analogous result for the class of strongly chordal split graphs~\cite{MR3853109}.
Also, the twin-width of chordal graphs can be arbitrarily large, as follows from the fact that this parameter is unbounded in the class of interval graphs, that is, intersection graphs of closed intervals on the real line~\cite{MR4262551}.
Split graphs, strongly chordal graphs, and interval graphs are all subclasses of the class of chordal graphs (see, e.g.,~\cite{MR1686154}).

Thus, one cannot obtain efficient algorithms for the class of chordal graphs by applying any of the algorithmic metatheorems for the aforementioned width parameters.
In fact, this is not surprising, as several problems remain \NP-complete in the class of chordal graphs, including the \textsc{Dominating Set}~\cite{MR754426} and the \textsc{Weighted Independent Dominating Set} problems~\cite{MR2087898}.
Nevertheless, many \NP-complete problems are known to be efficiently solvable for chordal graphs, including the \textsc{Independent Dominating Set} problem~\cite{MR687354}, as well as the \textsc{Independent Set}, the \textsc{Clique}, and the \textsc{Chromatic Number} problems, as a consequence of the fact that these problems are solvable in polynomial time in the more general class of perfect graphs~\cite{MR936633}.
Furthermore, since chordal graphs have only polynomially many minimal separators (see, e.g.,~\cite{MR1935387}), a metatheorem for such graph classes due to Fomin, Todinca, and Villanger~\cite{MR3311877} applies, establishing polynomial-time solvability of certain problems expressible in counting monadic second order logic.
Besides the already mentioned \textsc{Independent Set} problem, this approach captures the \textsc{Feedback Vertex Set} problem, the \textsc{Induced Matching} problem,  and various packing problems.

\subsection{Treewidth versus clique number}

Chordal graphs have an interesting property regarding treewidth.
While graphs with large cliques necessarily have large treewidth, in chordal graphs the converse holds, too: the treewidth of a chordal graph can only be large due to the presence of a large clique (see, e.g.,~\cite{MR1647486}).
Recently, we initiated in~\cite{DMS-WG2020,dallard2021treewidth} a systematic study of graph classes in which this sufficient condition for large treewidth---the presence of a large clique---is also necessary, which we call $(\tw,\omega)$-bounded.
A graph class $\G$ is said to be \emph{$(\tw,\omega)$-bounded} if it admits a  \emph{$(\tw,\omega)$-binding function}, that is, a function $f$ such that the treewidth of any graph $G\in \G$ is at most $f(\omega(G))$, where $\omega(G)$ is the clique number of $G$, and the same holds for all induced subgraphs of $G$.
The family of $(\tw,\omega)$-bounded graph classes provides a unifying framework for a variety of very different families of graph classes studied in the literature.
Besides graph classes of bounded treewidth, $(\tw,\omega)$-boundedness holds for graph classes of bounded independence number (as a consequence of Ramsey's theorem), intersection graphs of connected subgraphs of graphs with treewidth at most $t$, for any fixed positive integer $t$, and classes of graphs in which all minimal separators are of bounded size.
Intersection graphs of connected subgraphs of graphs with treewidth at most $t$ were studied in 1990 by Scheffler~\cite{MR1090614} and in 1998 by Bodlaender, Gustedt, and Telle~\cite{MR1642971}.
They include chordal graphs (for which $t = 1$, since every chordal graph is the intersection graph of subtrees in a tree~\cite{MR357218,MR505894,MR332541}) and circular-arc graphs, that is, intersection graphs of circular arcs on a circle (for which $t = 2$), as well as \emph{{$H$}-graphs}, that is, the intersection graphs of connected subgraphs of a subdivision of a fixed multigraph $H$, introduced in 1992 by B\'{\i}r\'{o}, Hujter, and Tuza~\cite{MR1172354} and studied more recently in a number of papers~\cite{DBLP:journals/endm/ChaplickZ17,MR4249058,MR3746153,MR4141534,MR4332111}.
Classes of graphs in which all minimal separators are of bounded size were studied in 1999 by Skodinis \cite{MR1852483}.
In~\cite{DMS-WG2020,dallard2021treewidth}, we
characterized, for each of six well-known graph containment relations (the minor, topological subgraph, subgraph, and their induced variants), the graphs $H$ such that the class of graphs excluding $H$ with respect to the relation is $(\tw,\omega)$-bounded.

In view of this richness of the family of $(\tw,\omega)$-bounded graph classes, one should not expect algorithmic metatheorems for $(\tw,\omega)$-bounded graph classes of similar generality as those available, for example, for graphs of bounded treewidth or bounded clique-width.
In fact, as already exemplified via the class of chordal graphs, several $\NP$-complete graph problems remain $\NP$-complete when restricted to some $(\tw,\omega)$-bounded graph classes.
This is the case not only for the \textsc{Dominating Set} and \textsc{Weighted Independent Dominating Set} problems, but also for the \textsc{Chromatic Number} problem, which is $\NP$-complete in the class of circular-arc graphs~\cite{MR578325}, and for the unweighted \textsc{Independent Dominating Set} problem, which is $\NP$-complete in the  $(\tw,\omega)$-bounded class of graphs that have a clique such that each vertex is either in the clique or is adjacent to at most one vertex not in the clique~\cite{MR2207507}.
Nevertheless, $(\tw,\omega)$-bounded graph classes do enjoy some good algorithmic properties related to clique and coloring problems.
They can be summarized as follows.
\begin{enumerate}
\item Chaplick and Zeman showed in~\cite{DBLP:journals/endm/ChaplickZ17} (see also~\cite{MR4332111}) that for every positive integer $k$, any $(\tw,\omega)$-bounded graph class $\mathcal{G}$ having a computable binding function admits linear-time algorithms for the {\sc $k$-Clique} and the {\sc List $k$-Coloring} problems.
Furthermore, as we showed in~\cite{DMS-WG2020,dallard2021treewidth}, any such class admits a polynomial-time algorithm for the {\sc List $k$-Coloring} problem that is \emph{robust} in the sense of Raghavan and Spinrad~\cite{MR2006100}: it either solves the problem or determines that the input graph is not in $\mathcal{G}$.

\item Similarly, for every positive integer $k$, there exists a robust quadratic-time algorithm for the {\sc List $k$-Edge-Coloring} problem in any $(\tw,\omega)$-bounded graph class having a computable binding function $f$.
Indeed, if the input graph $G$ has a vertex with degree more than $k$, then it cannot be $k$-edge-colorable with respect to the given lists of colors on the edges, while if the maximum degree is bounded by $k$, then the clique number is bounded by $k+1$ and therefore the treewidth is bounded by $c_k = \max\{f(j) : 1\le j\le f(k+1)\}$.
In this case, we can use the linear-time algorithm of Bodlaender~\cite{MR1417901} to test whether $\tw(G)\le c_k$.
If $\tw(G)>c_k$, then the graph is not in the class.
If $\tw(G)\le c_k$, then the clique-width of $G$ is at most $\ell$ where $\ell = 3\cdot 2^{c_k-1}$ (see Corneil and Rotics~\cite{MR2148860}), and using a $(2^{2\ell+1}-1)$-expression of $G$ that can be computed in time $\mathcal{O}(|V(G)|^2)$
with an algorithm due to Fomin and Korhonen~\cite{FominKorhonen2022}, the {\sc List $k$-Edge-Coloring} problem can be solved in time $\mathcal{O}(|V(G)|^2)$ by invoking a result of Kobler and Rotics~\cite{MR1948213}.

\item As also shown in~\cite{DMS-WG2020,dallard2021treewidth},
known approximation algorithms for treewidth (see, e.g.,~\cite{MR2411037}) lead to improved approximations for the clique number of a graph from a $(\tw,\omega)$-bounded graph class having a computable exponential binding function.
The approximation can be improved further if the binding function is computable and either linear or polynomial.
For example, in the case of a polynomial $(\tw,\omega)$-binding function \hbox{$f(k) = \mathcal{O}(k^c)$} for some constant $c$, the clique number can be approximated in polynomial time to within a factor of ${\sf opt}^{1-1/(c+\varepsilon)}$ for all $\varepsilon>0$, where {\sf opt} denotes the maximum size of a clique of the input graph $G$.
For general graphs, this problem is notoriously difficult to approximate: for every $\varepsilon>0$, there is no polynomial-time algorithm for approximating the maximum clique in an $n$-vertex graph to within a factor of $n^{1-\varepsilon}$ unless $\P = \NP$ \cite{MR2403018}.
\end{enumerate}

An interesting open problem is whether $(\tw,\omega)$-boundedness has any further algorithmic implications, for example for problems related to independent sets.
An \emph{independent set}---or \emph{stable set}---in a (simple, finite, and undirected) graph $G$ is a set of pairwise non-adjacent vertices.
The \emph{independence number} of a graph $G$ is the maximum size of an independent set in $G$.
More generally, given a graph $G$ and a weight function $w:V(G)\to \mathbb{Q}_+$, the \textsc{Max Weight Independent Set} problem asks to find an independent set $I$ in $G$ of maximum possible weight $w(I)$, where $w(I) = \sum_{x\in I}w(x)$.
The \textsc{Max Weight Independent Set} problem is one of the most studied graph optimization problems.
It is not only (strongly) \NP-hard in general~\cite{MR0378476} but also \NP-hard to approximate on $n$-vertex graphs to within a factor of $n^{1-\epsilon}$ for every $\epsilon>0$~\cite{MR2403018}.
Understanding what restrictions on the input graph make the problem more tractable has been a subject of investigation for decades, with several recent developments~\cite{MR4117301,MR4141321,MR3763297,MR3909546,MR4262549,MR4262487,MR4232071,DBLP:conf/sosa/PilipczukPR21,10.1145/3406325.3451034}.

We observed in~\cite{dallard2021treewidth} that in all the $(\tw,\omega)$-bounded graph classes defined by excluding a single graph $H$ with respect to any of the six considered graph containment relations (the minor, topological subgraph, subgraph, and their induced variants), the \textsc{Max Weight Independent Set} problem can be solved in polynomial time, except possibly for the case when $H$ is excluded with respect to the induced minor relation and $H$ is isomorphic to either $K_5^-$ (the complete $5$-vertex graph minus an edge), $W_4$ (a wheel with four spokes, that is, the graph obtained from the $4$-vertex cycle by adding to it a universal vertex), or the complete bipartite graph $K_{2,q}$ where $q\ge 3$.
(See \cref{sec:preliminaries} for precise definitions.)
We asked whether there is a $(\tw,\omega)$-bounded graph class in which the \textsc{Max Weight Independent Set} problem is \NP-hard.
Unless $\P = \NP$, the question can be equivalently formulated as follows.

\begin{question}\label{question3}
Is the \textsc{Max Weight Independent Set} problem solvable in polynomial time in every $(\tw,\omega)$-bounded graph class?
\end{question}

\subsection{Our results}

\begin{sloppypar}
Motivated by the quest to increase our understanding of algorithmic implications of \hbox{$(\tw,\omega)$-boundedness}, we provide in this work a partial answer to \cref{question3} by identifying a sufficient condition for polynomial-time solvability of the \textsc{Max Weight Independent Set} problem in a $(\tw,\omega)$-bounded graph class.
In fact, we do this in the more general context of the \textsc{Max Weight Independent Packing} problem.
This problem takes as input a graph $G$ and a family $\HH$ of connected subgraphs of $G$, each equipped with a nonnegative weight, and the task is to compute a maximum-weight subfamily of subgraphs from $\HH$ such that no two of them share a vertex or are connected by an edge.
\end{sloppypar}

In order to explain the corresponding result, we need some additional definitions.
In particular, we introduce a new min-max graph parameter related to tree decompositions.
Given a tree decomposition $\mathcal{T}$ of a graph $G$, we define the \emph{independence number} of $\mathcal{T}$ as the maximum independence number over all subgraphs of $G$ induced by some bag of $\mathcal{T}$.
The \emph{tree-independence number} of a graph $G$ is then defined as the minimum independence number over all tree decompositions of $G$.
We say that a graph class $\mathcal{G}$ has \emph{bounded tree-independence number} if there exists an integer $k$ such that every graph in $\mathcal{G}$ has tree-independence number at most $k$.
Furthermore, tree decompositions with bounded independence number are said to be \emph{$\alpha$-bounded} and we say that the (boundedness of the) tree-independence number in a graph class $\mathcal G$ is \emph{efficiently witnessed} if there exists a polynomial-time algorithm to compute an $\alpha$-bounded tree decomposition for every graph in $\mathcal G$.

\medskip

We show that graph classes with bounded tree-independence number satisfy the following properties:

\begin{enumerate}[label=(\arabic*)]
\item In every graph class with bounded tree-independence number, the treewidth is bounded by an explicit polynomial function of the clique number (see \cref{bounded tin implies bounded tw-omega}); the degree of the polynomial equals the bound on the tree-independence number.
Consequently, every graph class with bounded tree-independence number is \hbox{$(\tw,\omega)$-bounded}, admits linear-time algorithms for the {\sc $k$-Clique} and the {\sc List $k$-Coloring} problems for every positive integer $k$, as well as a polynomial-time robust algorithm for the {\sc List $k$-Coloring} and {\sc List $k$-Edge-Coloring} problems, and an improved approximation for the clique number.

\item In every graph class with bounded tree-independence number, the \textsc{Max Weight Independent Packing} problem
can be solved in polynomial time provided that the input graph $G$ is given along with an $\alpha$-bounded tree decomposition.
(See \cref{thm:bounded-tree-independence-number-packings}.)
This implies analogous results for several problems studied in the literature, including the \textsc{Independent $\mathcal{F}$-Packing}, the \textsc{Max Weight Independent Set}, the \textsc{Max  Weight Induced Matching}, the \textsc{Dissociation Set}, and the \textsc{$k$-Separator} problems.
We refer to \cref{sec:implications} for more details and references.
\end{enumerate}
As part of our approach related to solving the
\textsc{Max Weight Independent Packing} problem in graph classes with efficiently witnessed tree-independence numbers, we show that, given any graph $G$ and a family $\HH$ of connected subgraphs of $G$, the tree-independence number of a certain derived graph $G(\HH)$ cannot exceed the tree-independence number of $G$ (see \cref{tin-can-only-go-down-when-computing-G(HH)}).
Here, the graph $G(\HH)$ has the subgraphs in $\HH$ as vertices, with two adjacent if and only if they share a vertex or there is an edge connecting them.
Cameron and Hell showed in~\cite{MR2190818} that the following graph classes are closed under the transformation $G\mapsto G(\mathcal{H})$ (for any $\mathcal{H}$): chordal graphs, interval graphs, circular-arc graphs, polygon-circle graphs, interval filament graphs, cocomparability graphs, AT-free graphs, and weakly chordal graphs.
Our result adds the classes of graphs with tree-independence number at most~$k$ (for any positive integer $k$) to the list of graph classes that are known to be closed under this transformation, generalizing the result for chordal graphs, which, as we will see, are exactly the graphs with tree-independence number at most one.

The original motivation for introducing the tree-independence number stems from the fact that in a tree decomposition of constant width each bag interacts with an optimal solution to the problem under consideration in only a bounded number of ways, which can be enumerated efficiently.
This often leads to efficient dynamic programming algorithms.
It is natural to relax the condition on the width to the requirement that for each bag there are only \emph{polynomially many} ways in which an optimal solution can interact with the bag, and that these interactions can be enumerated efficiently.
In particular, this happens whenever each bag can interact with an optimal solution in a bounded number of vertices. (This approach was independently suggested by Maria Chudnovsky~\cite{Chudnovsky2020StonyBrook,Chudnovsky2021Berlin,Chudnovsky2021Prague,Chudnovsky2021IWOCA}.)
The definition of the independence number of a tree decomposition is in line with this discussion: it gives an upper bound on the number of vertices that an optimal solution to the \textsc{Max Weight Independent Set} problem can contain from any given bag.
Thus, we are interested in tree decompositions of $G$ that are not necessarily of smallest possible width, but that have the smallest possible independence number.
This immediately leads to the definition of the tree-independence number.

The algorithmic result about the \textsc{Max Weight Independent Packing} problem raises two natural questions:
\begin{question}\label{question1}
Which graph classes have bounded tree-independence number?
\end{question}
\begin{question}\label{question2}
For which graph classes can an $\alpha$-bounded tree decomposition be computed in polynomial time?
\end{question}

The wide applicability of our algorithmic result is indicated by the following two partial answers to \cref{question1,question2}.

First, the following families of graph classes all have bounded tree-independence number: graph classes of bounded treewidth, graph classes of bounded independence number, classes of intersection graphs of connected subgraphs of graphs with bounded treewidth, and classes of graphs in which all minimal separators are of bounded size (see \cref{sufficient conditions for bounded tree-alpha}).

Second, when a single graph $H$ is excluded with respect to any of six well-known graph containment relations (the minor, topological subgraph, subgraph, and their induced variants), the resulting graph class is $(\tw,\omega)$-bounded if and only if it has bounded tree-independence number (see the third paper of the series~\cite{dallard2022secondpaper}).
In each such case there is a polynomial-time algorithm for computing an $\alpha$-bounded tree decomposition.
These results are obtained via characterizations, for each of six graph containment relations, of the graphs $H$ such that the class of graphs excluding $H$ has bounded tree-independence number.
For completeness, we recall those characterizations below.

A graph is \emph{subcubic} if every vertex of it has degree at most three. 
We denote by $\mathcal{S}$ the class of forests in which every connected component has at most three leaves (vertices of degree one).

\begin{theorem}[Dallard, Milani\v{c}, and \v{S}torgel~\cite{dallard2022secondpaper}]\label{dichotomy tin}
   For every graph $H$, the following statements hold.
    \begin{enumerate}
        \item The class of $H$-subgraph-free graphs has bounded tree-independence number if and only if $H \in \mathcal{S}$.

        \item The class of $H$-topological-minor-free graphs has bounded tree-independence number if and only if $H$ is subcubic and planar.

        \item The class of $H$-minor-free graphs has bounded tree-independence number if and only if $H$ is planar.
        \item The class of $H$-free graphs has bounded tree-independence number if and only if $H$ is either an induced subgraph of $P_3$ or an edgeless graph.

        \item The class of $H$-induced-topological-minor-free graphs has bounded tree-independence number if and only if $H$ is either an induced topological minor of $C_4$ or $K_4^-$, or $H$ is edgeless.

        \item The class of $H$-induced-minor-free graphs has bounded tree-independence number if and only if $H$ is an induced minor of $W_4$, $K_5^-$, or $K_{2,q}$ for some $q \in \mathbb{Z}_+$, or $H$ is edgeless.
    \end{enumerate}
\end{theorem}

In particular, this result leads to a new infinite family of graph classes in which the \textsc{Max Weight Independent Set} problem is solvable in polynomial time.
We refer to~\cite{dallard2022secondpaper} for details.

Since the tree-independence number is a newly introduced graph invariant, we also provide a set of results about its basic properties (see \cref{sec:basic}).
These include a characterization of graphs with tree-independence number at most one (as exactly the chordal graphs), the \NP-hardness of the problem of computing the tree-independence number, two sharp upper bounds in terms of the independence number and the treewidth, the monotonicity under vertex deletions and edge contractions, and the behavior under clique cutsets.

\medskip
\noindent\textbf{Note.}
After the paper was submitted for publication, we became aware of the work of Yolov~\cite{MR3775804}, who considered two parameters based on tree decompositions, including the tree-independence number (named \emph{$\alpha$-treewidth} and denoted by $\alpha$-$\tw$).
He gave an algorithm that, given an $n$-vertex graph $G$ and an integer $k$, in time $n^{\mathcal{O}(k^3)}$ either computes a tree decomposition of $G$ with independence number $\mathcal{O}(k^3)$ or correctly determines that the tree-independence number of $G$ is larger than $k$.
In particular, this means that the tree-independence number is efficiently witnessed whenever it is bounded.
Yolov's result was recently improved further by Dallard, Golovach, Fomin, Korhonen, and Milani{\v c}~\cite{dallard2022computing}, who improved the running time to $2^{\mathcal{O}(k^2)} n^{\mathcal{O}(k)}$ and the upper bound on the independence number of the computed tree decomposition to $8k$.

\subsection{Overview of our methodology}

We now summarize the key ideas leading to our main results.

We develop polynomial-time algorithms for the \textsc{Max Weight Independent Packing} problem in graph classes with bounded and efficiently witnessed tree-independence number in two steps.
First, we develop a dynamic programming algorithm for the \textsc{Max Weight Independent Set} problem, assuming that the input graph is equipped with a tree decomposition of bounded independence number (\cref{thm:bounded-tree-independence-number}).
We do this by adapting the standard dynamic programming approach for the \textsc{Max Weight Independent Set} problem over tree decompositions of bounded width to the more general setting of tree decompositions with bounded independence number.
We in fact derive an improved version of this algorithmic result by introducing the concept of an \emph{$\ell$-refined tree decomposition}, which is a tree decomposition in which each bag comes equipped with a subset of at most $\ell$ vertices (see \cref{sec:ell-refined}).
Ignoring such vertices when computing the independence number of the subgraph induced by the bag may lead to better running times
(\cref{thm:bounded-ell-refined-tree-independence-number}).
Two particular examples of such improvements are given in~\cite{dallard2022secondpaper}.
Second, we show that the \textsc{Max Weight Independent Packing} problem can be solved in polynomial time if the input graph is equipped with an $\alpha$-bounded tree decomposition (\cref{thm:bounded-tree-independence-number-packings}).
We do this by reducing the problem to the \textsc{Max Weight Independent Set} problem in the derived graph $G(\HH)$.
More specifically, we show that in this case one can compute in polynomial time both the graph $G(\HH)$ as well as a tree decomposition $\mathcal{T}'$ of $G(\HH)$ such that the independence number of $\mathcal{T}'$ is at most $k$, where $k$ is a fixed upper bound on the independence number of the tree decomposition $\mathcal{T}$ of the input graph $G$ (\cref{tin-of-G(H)-algorithmic}).

\subsection{Relation to previous work}

We asked in~\cite{DMS-WG2020,dallard2021treewidth} whether every $(\tw,\omega)$-bounded graph class is in fact \emph{polynomially $(\tw,\omega)$-bounded}, in the sense that the treewidth of graphs in the class and all their induced subgraphs is bounded by a fixed \emph{polynomial} function of the clique number.
This is indeed the case for graph classes with bounded tree-independence number (\cref{bounded tin implies bounded tw-omega}).
In general, the notion of the tree-independence number may give some insight into the properties of a possible counterexample to the question about the equality of $(\tw,\omega)$-boundedness and polynomial $(\tw,\omega)$-boundedness: any $(\tw,\omega)$-bounded graph class that is not polynomially $(\tw,\omega)$-bounded has to contain graphs with arbitrarily large tree-independence number.

\begin{sloppypar}
The concept of tree-independence number is also related to $\chi$-boundedness.
In the late 1980s, Gy\'arf\'as introduced the concept of $\chi$-bounded graph classes in order to generalize perfection~\cite{MR951359}.
The \emph{chromatic number} of $G$ is the smallest positive integer $k$ such that $G$ is $k$-colorable (that is, its vertex set can be expressed as the union of $k$ independent sets) and a graph class $\mathcal{G}$ is said to be \emph{$\chi$-bounded} if the chromatic number of the graphs in the class and all their induced subgraphs is bounded from above by some function of the clique number.
The $\chi$-bounded graph classes were studied extensively in the literature (see~\cite{MR4174126} for a survey).
Every $(\tw,\omega)$-bounded graph class is $\chi$-bounded (see, e.g.,~\cite{dallard2021treewidth}).
Esperet asked whether every $\chi$-bounded graph class is \emph{polynomially $\chi$-bounded}, that is, the chromatic number of the graphs in the class and all their induced subgraphs is bounded by a polynomial function of the clique number (see~\cite{esperet2017graph}).
For general classes, Esperet's question was recently answered in the negative by Bria\'nski, Davies, and Walczak~\cite{brianski2023separating}.
Nonetheless, since the chromatic number of any graph is bounded from above by its treewidth plus one, we obtain the following.
\end{sloppypar}

\begin{fact}
Graph classes with bounded tree-independence number are polynomially $\chi$-bounded.
\end{fact}

Our result on the polynomial-time solvability of the \textsc{Max Weight Independent Set} problem in classes of graphs equipped with a tree decomposition of bounded independence number (\cref{thm:bounded-tree-independence-number}) generalizes a result of Krause who showed in~\cite{DBLP:conf/cc/Krause13} that the \textsc{Max Independent Set} problem is solvable in polynomial time on intersection graphs of connected subgraphs of graphs of bounded treewidth (assuming that the input graph is given with an intersection model).
In some sense, our algorithmic result for the \textsc{Max Weight Independent Packing} problem (\cref{thm:bounded-tree-independence-number-packings}) generalizes the polynomial-time solvability of the \textsc{Max Induced Matching} problem in the class of chordal graphs~\cite{MR1011265}, of the \textsc{$k$-Separator} problem (for fixed $k$) in the classes of interval graphs and graphs of bounded treewidth~\cite{MR3296270}, and of the \textsc{Independent $\mathcal{F}$-Packing} problem in the classes of chordal and circular-arc graphs~\cite{MR2190818}.

On a more general note, let us remark that the tree-independence number shares the property of several graph invariants studied in the literature in the sense that they are defined by the existence of a tree decomposition with certain properties of the bags.
These include connectivity properties (connected treewidth~\cite{MR3800845}, also known as bag-connected treewidth~\cite{jegou2014tree}), metric properties (tree-length~\cite{MR2326163} and tree-breadth~\cite{MR3209756}), and chromatic properties (tree-chromatic number~\cite{MR3425243}).

\subsection{Organization of the paper}

In \cref{sec:preliminaries}, we provide the necessary preliminaries.
In \cref{sec:basic}, we derive several basic properties of the tree-independence number, including the \NP-hardness of the problem of computing it.
In \cref{sec:ell-refined}, we introduce the concepts of $\ell$-refined tree decompositions and their residual independence numbers.
In \cref{sec:MWIS}, we prove that if $\mathcal{G}$ is a graph class with bounded and efficiently witnessed tree-independence number, then there exists a polynomial-time algorithm for the \textsc{Max Weight Independent Set} problem when the input graph is restricted to belong to $\mathcal{G}$.
We show in \cref{sec:packing} that for any graph $G$ and any family $\HH$ of connected subgraphs of $G$, the derived graph $G(\HH)$ has tree-independence number bounded by that of $G$.
This allows us to generalize the algorithmic result of \cref{sec:MWIS} by developing, in \cref{sec:implications}, an efficient algorithm for the \textsc{Max Weight Independent Packing} problem, provided that the input graph $G$ is given along with an $\alpha$-bounded tree decomposition.
We conclude the paper with some open questions in \cref{sec:open-questions}.

\section{Preliminaries}\label{sec:preliminaries}

Throughout the paper, we denote by $\mathbb{Z}_+$ the set of nonnegative integers and by $\mathbb{Q}_+$ the set of nonnegative rational numbers.
We assume familiarity with the basic concepts in graph theory as used, e.g., by West~\cite{MR1367739}.
We denote the vertex set and the edge set of a graph $G$ by $V(G)$ and $E(G)$, respectively.
A graph is \emph{null} if it has no vertices.
The \emph{neighborhood} of a vertex $v$ in $G$, which corresponds to the set of vertices adjacent to $v$ in $G$, is denoted by $N_G(v)$.
The \emph{closed neighborhood} of $v$ is the set $N_G[v] = N_G(v) \cup \{v\}$.
These concepts are extended to sets $X\subseteq V(G)$ so that $N_G[X]$ is defined as the union of all closed neighborhoods of vertices in $X$, and $N_G(X)$ is defined as the set $N_G[X]\setminus X$.
The \emph{degree} of $v$, denoted by $d_G(v)$, is the cardinality of the set $N_G(v)$.
When there is no ambiguity, we may omit the subscript $G$ in the notations of the degree, and open and closed neighborhoods, and thus simply write $d(v)$, $N(v)$, and $N[v]$, respectively.
Given a set $X \subseteq V(G)$, we denote by $G-X$ the graph obtained from $G$ after deleting all the vertices in $X$, and by $G[X]$ the subgraph of $G$ induced by $X$, that is, the graph $G-(V(G)\setminus X)$.
Similarly, given a vertex $v \in V(G)$, we denote by $G-v$ the graph obtained from $G$ after deleting $v$.
The fact that two graphs $G$ and $H$ are isomorphic to each other is denoted by $G\cong H$.
For a positive integer $n$, we denote the $n$-vertex complete graph, path, and cycle by $K_n$, $P_n$, and $C_n$, respectively.
Similarly, for $n\ge 2$ we denote by $K_n^{-}$ the graph obtained from the complete graph $K_n$ by deleting an edge, and for positive integers $m$ and $n$ we denote by $K_{m,n}$ the complete bipartite graph with parts of sizes $m$ and $n$.
The graph $K_4^-$ is also called the \emph{diamond}.
For $n\ge 4$, we denote by $W_n$ the graph obtained from a cycle $C_n$ by adding to it a universal vertex (that is, a vertex adjacent to all vertices of the $C_n$).

An \emph{independent set} in a graph $G$ is a set of pairwise non-adjacent vertices, and a \emph{clique} is a set of pairwise adjacent vertices.
The \emph{independence number} of $G$, denoted by $\alpha(G)$, is the maximum size of an independent set in~$G$.
The \emph{clique number} of $G$, denoted by $\omega(G)$, is the maximum size of a clique in~$G$.
Given two non-adjacent vertices $u$ and $v$ in $G$, a \emph{$u{,}v$-separator} is a set $S\subseteq V(G)\setminus \{u,v\}$ such that $u$ and $v$ belong to different components of $G-S$.
A $u{,}v$-separator $S$ is \emph{minimal} if no proper subset of $S$ is a $u{,}v$-separator.
A \emph{minimal separator} in $G$ is a set $S\subseteq V(G)$ that is a minimal $u{,}v$-separator for some non-adjacent vertex pair $u{,}v$.

A graph $H$ is said to be an \emph{induced subgraph} of $G$ if $H$ can be obtained from $G$ by deleting vertices. 
If $H$ is obtained from $G$ by deleting vertices and edges, then $H$ is a \emph{subgraph} of $G$.
The \emph{subdivision of an edge} $uv$ of a graph is the operation that deletes the edge $uv$ and adds a new vertex $w$ and two edges $uw$ and $wv$.
A \emph{subdivision of a graph} $H$ is a graph obtained from $H$ by a sequence of edge subdivisions.
A graph $H$ is said to be a \emph{topological minor} of a graph $G$ if $G$ contains a subdivision of $H$ as a subgraph.
Similarly, $H$ is an \emph{induced topological minor} of $G$ if some subdivision of $H$ is isomorphic to an induced subgraph of $G$.
An \emph{edge contraction} is the operation of deleting a pair of adjacent vertices and replacing them with a new vertex whose neighborhood is the union of the neighborhoods of the two original vertices.
If $H$ can be obtained from $G$ by a sequence of vertex deletions, edge deletions, and edge contractions, then $H$ is said to be a \emph{minor} of $G$. 
Finally, we say that $G$ contains $H$ as an \emph{induced minor} if $H$ can be obtained from $G$ by a sequence of vertex deletions and edge contractions.
If $G$ does not contain an induced subgraph isomorphic to $H$, then we say that $G$ is \emph{$H$-free}. 
Analogously, we also say that $G$ is $H$-subgraph-free, $H$-topological-minor-free, $H$-induced-topological-minor-free, $H$-minor-free, or $H$-induced-minor-free, respectively, for the other five graph containment relations.

A \emph{tree decomposition} of a graph $G$ is a pair $\mathcal{T} = (T, \{X_t\}_{t\in V(T)})$ where $T$ is a tree and every node $t$ of $T$ is assigned a vertex subset $X_t\subseteq V(G)$ called a bag such that the following conditions are satisfied:
every vertex is in at least one bag, for every edge $uv\in E(G)$ there exists a node $t\in V(T)$ such that $X_t$ contains both $u$ and $v$, and for every vertex $u\in V(G)$ the subgraph $T_u$ of $T$ induced by the set $\{t\in V(T): u\in X_t\}$ is connected (that is, a tree).
The \emph{width} of $\mathcal{T}$, denoted by ${\it width}(\mathcal{T})$, is the maximum value of $|X_t|-1$ over all $t\in V(T)$.
The \emph{treewidth} of a graph $G$, denoted by $\tw(G)$, is defined as the minimum width of a tree decomposition of $G$.

We next state two useful properties of tree decompositions.

\begin{lemma}[Scheffler~\cite{MR1090614}, Bodlaender and M\"ohring~\cite{MR1215226}]\label{clique in bag}
Let $G$ be a graph, let $\mathcal{T} = (T, \{X_t\}_{t\in V(T)})$ be a tree decomposition of $G$, and let $C$ be a clique in $G$.
Then there exists a bag $X_t$ such that $C\subseteq X_t$.
\end{lemma}

\begin{corollary}\label{clique-tw-bound}
Every graph $G$ satisfies $\tw(G)\ge \omega(G)-1$.
\end{corollary}

\begin{lemma}\label{conformality-of-bags}
Let $\mathcal{T} = (T,\{X_t\}_{t\in V(T)})$ be a tree decomposition of a graph $G$.
Then for every set $S\subseteq V(G)$ such that every pair of vertices in $S$ is contained in a bag of $\mathcal{T}$, there exists a bag $X_t$ such that $S\subseteq X_t$.
\end{lemma}

\begin{proof}
Suppose that every pair of vertices of $S$ is contained in a bag of $\mathcal{T}$. For each $v\in S$, the set of nodes of $T$ labeled by bags containing $v$ induces a subtree $T_v$ of $T$. Since every two vertices in $S$ are contained in a bag of $\mathcal{T}$, every two of the subtrees
in $\{T_v: v\in S\}$ have a node in common.
It is known (and easy to see) that any family of node sets of subtrees of a tree has the Helly property, and thus there exists a node $t\in V(T)$ common to all the trees $T_v$, $v\in S$. We infer that $v\in X_t$ for all $v\in S$, that is, $S\subseteq X_t$.
\end{proof}

There are many equivalent characterizations of treewidth, including one in terms of chordal graphs.
A graph $G$ is \emph{chordal} if it has no induced cycles of length at least four.
Given a graph $G = (V,E)$, a \emph{triangulation} of $G$ is a chordal graph of the form
$(V,E\cup F)$ where $F$ is a subset of the set of non-edges of $G$.
A triangulation $(V,E \cup F)$ of $G$ is a \emph{minimal triangulation} of $G$ if $G$ has no triangulation $(V,E\cup F')$ such that $F'$ is a proper subset of $F$.

\begin{theorem}[Robertson and Seymour \cite{MR855559}]\label{tw-via-triangulations}
For every graph $G$, we have
\[\tw(G) = \min\{\omega(H)-1 : H \textrm{ is a triangulation of } G\}\,.\]
\end{theorem}

Our approach in proving the \NP-hardness of computing the tree-independence number in \cref{sec:basic} relies on the fact that every tree decomposition has a bag containing the closed neighborhood of some vertex (see \cref{minimum-degree}).
This property of tree decompositions implies the inequality $\tw(G)\ge \delta(G)$, where $\delta(G)$ denotes the minimum degree of a vertex in $G$, valid for any graph $G$ (see, e.g., \cite{MR2829452}).

In order to prove \cref{minimum-degree}, we first recall some of the many characterizations of chordal graphs.
A \emph{clique tree} of a graph $G$ is a tree decomposition of $G$ such that the bags are exactly the maximal cliques of $G$.
Given a collection $\{T_1,\ldots, T_n\}$ of subtrees in a tree $T$, the \emph{intersection graph} of $\{T_1,\ldots, T_n\}$ is the graph with vertex set $\{1,\ldots, n\}$, in which two distinct vertices $i$ and $j$ are adjacent if and only if $T_i$ and $T_j$ have a vertex in common.
Given a graph $G$, a \emph{simplicial vertex} in $G$ is a vertex whose neighborhood is a clique.
The following theorem combines results from several sources~\cite{MR130190,MR357218,MR505894,MR332541,MR1320296}.

\begin{theorem}\label{chordal-clique-tree}
Let $G$ be a graph. Then, the following conditions are equivalent:
\begin{enumerate}
\item $G$ is a chordal graph.
\item $G$ has a clique tree.
\item $G$ is the intersection graph of subtrees in a tree.
\item Every nonnull induced subgraph of $G$ has a simplicial vertex.
\end{enumerate}
\end{theorem}

\begin{lemma}\label{minimum-degree}
Let $G$ be a graph and let $\mathcal{T} = (T, \{X_t\}_{t\in V(T)})$ be a tree decomposition of $G$.
Then there exists a vertex $v\in V(G)$ and a node $t\in V(T)$ such that $N[v]\subseteq X_t$.
\end{lemma}

\begin{proof}
Let $G'$ be the graph with vertex set $V(G)$ such that two distinct vertices $u$ and $v$ are adjacent in $G'$ if and only if there exists a bag $X_t$ of $\mathcal T$ with $u,v \in X_t$.
Note that for every vertex $v \in V(G)$ it holds that $N_G[v] \subseteq N_{G'}[v]$.
Recall that for each vertex $v\in V(G')$, the set of nodes $t\in V(T)$ such that $v\in X_t$ induces a subtree $T_v$ of $T$.
Thus, two distinct vertices $u$ and $v$ of $G'$ are adjacent if and only if the corresponding trees $T_u$ and $T_v$ have a node in common.
This means that $G'$ is the intersection graph of subtrees in a tree and hence, by \cref{chordal-clique-tree}, $G'$ is a chordal graph.
Observe that $\mathcal T$ is also a tree decomposition of $G'$.
By \cref{chordal-clique-tree}, $G'$ has a simplicial vertex $v$, and hence by \cref{clique in bag}, there must exist a node $t \in V(T)$ such that $N_{G'}[v] \subseteq X_t$.
Therefore, $N_G[v] \subseteq X_t$, which concludes the proof.
\end{proof}

Let us remark that \cref{minimum-degree} can also be derived from the proofs of Lemmas~2 and 4 in~\cite{MR2829452} (which do not rely on connections with chordal graphs).

An (integer) \emph{graph invariant} is a mapping from the class of all graphs to the set of nonnegative integers that does not distinguish between isomorphic graphs.
For a graph invariant $\rho$, we say that a graph class $\mathcal{G}$ has \emph{bounded $\rho$} if there exists an integer $k$ such that $\rho(G)\le k$ for every graph $G$ in the class.
A graph class $\mathcal{G}$ is said to be \emph{$(\tw,\omega)$-bounded} if it admits a \emph{$(\tw,\omega)$-binding function}, that is, a function $f$ such that for every graph $G$ in the class and any induced subgraph $G'$ of $G$, the treewidth of $G'$ is at most $f(\omega(G'))$.
Ramsey's theorem~\cite{MR1576401} states that for every two positive integers $p$ and $q$ there exists an integer $N(p,q)$ such that every graph with at least $N(p,q)$ vertices contains either a clique of size $p$ or an independent set of size $q$. The least such positive integer is denoted by $R(p,q)$.

\section{Definition and basic properties}\label{sec:basic}

We now define the two main concepts studied in this paper.

\begin{definition}
Consider a graph $G$ and a tree decomposition $\mathcal{T} = (T, \{X_t\}_{t\in V(T)})$ of $G$.
The \emph{independence number} of $\mathcal T$, denoted by $\alpha(\mathcal{T})$,
is defined as follows:
\[\alpha(\mathcal{T}) = \max_{t\in V(T)} \alpha(G[X_t])\,.\]
The \emph{tree-independence number} of $G$, denoted by $\tin(G)$, is the minimum independence number among all possible tree decompositions of $G$.
\end{definition}

In this section we derive several basic properties of the tree-independence number of graphs.
First, we observe that Ramsey's theorem implies that graph classes with bounded tree-independence number are $(\tw,\omega)$-bounded with a polynomial binding function.

\begin{lemma}\label{bounded tin implies bounded tw-omega}
For every positive integer $k$, the class of graphs with tree-independence number at most $k$ is $(\tw,\omega)$-bounded, with a binding function $f(p) = R(p+1,k+1)-2$, which is upper-bounded by a polynomial in $p$ of degree $k$.
\end{lemma}

\begin{proof}
    The standard proof of Ramsey's theorem is based on the inequality $R(p,q) \le  R(p-1, q) + R(p, q-1)$ for all $p,q\ge 2$, which implies that $R(p,q)\le \binom{p+q-2}{q-1}$ for all positive integers $p$ and $q$.
    For fixed $q$, this is a polynomial in $p$ of degree $q-1$.

    Let us now fix $p \in \mathbb{Z}_+$ and let $G$ be a graph such that $\omega(G) = p$ and $\tin(G) \leq k$.
    Fix a tree decomposition $\mathcal T$ of $G$ with independence number at most $k$.
    Note that every bag of $\mathcal T$ induces a subgraph of $G$ with independence number at most $k$ and clique number at most $p$.
    Thus, for every bag $X$ of $\mathcal T$, Ramsey's theorem implies that $|X| \leq R(p+1,k+1)-1$.
    It follows that $\tw(G) \leq R(p+1,k+1)-2$, as claimed.
\end{proof}

Next, we show that the tree-independence number can be used to characterize chordal graphs.

\begin{theorem}\label{chordal}
Let $G$ be a graph. Then $\tin(G)\le 1$ if and only if $G$ is chordal.
\end{theorem}

\begin{proof}
If $G$ is chordal, then, by \cref{chordal-clique-tree}, $G$ has a clique tree and thus $\tin(G)\le 1$.
Conversely, if $G$ has a tree decomposition $\mathcal T$ with independence number at most one, then every bag of $\mathcal T$ is a clique in $G$.
Since $\mathcal T$ is a tree decomposition, for every vertex $u$ of $G$ the subgraph $T_u$ of $T$ induced by the set $\{t\in V(T): u\in X_t\}$ is a tree.
Furthermore, since each bag of $\mathcal T$ is a clique, two distinct vertices $u$ and $v$ of $G$ are adjacent if and only if they belong to a same bag, which is in turn equivalent to the condition that $T_u$ and $T_v$ have a vertex in common.
Thus, $G$ is the intersection graph of the collection of subtrees $\{T_u : u\in V(G)\}$.
Applying \cref{chordal-clique-tree}, we conclude that $G$ is chordal.
\end{proof}

A tree decomposition of a graph $G$ is said to be \emph{trivial} if it has a single bag $X_t = V(G)$.
Next, we establish the \NP-hardness of the problem of computing the tree-independence number of a graph.
Our reduction is based on the following lemma, which we derive from \cref{minimum-degree}.

\begin{lemma}\label{lem:alpha}
Let $G$ be a graph and let $G'$ be the graph obtained from two disjoint copies of $G$ by adding all possible edges between them. Then $\tin(G') = \alpha(G)$.
\end{lemma}

\begin{proof}
Let us denote by $G_1$ and $G_2$ the two disjoint copies of $G$ such that $V(G') = V(G_1)\cup V(G_2)$.
Observe that every independent set in $G'$ is entirely contained in either $G_1$ or $G_2$ and hence, we have $\alpha(G') = \alpha(G)$.
Thus, the trivial tree decomposition of $G'$ has independence number equal to $\alpha(G)$ and consequently $\tin(G')\le \alpha(G)$.

For the converse direction, consider an arbitrary tree decomposition $\mathcal{T} = (T, \{X_t\}_{t\in V(T)})$ of $G'$.
By \cref{minimum-degree}, there exists a vertex $v\in V(G)$ and a node $t\in V(T)$ such that $N[v]\subseteq X_t$.
Assume without loss of generality that $v\in V(G_1)$.
Then $V(G_2)\subseteq N(v)\subseteq X_t$, and therefore $\alpha(G'[X_t])\ge \alpha(G_2) = \alpha(G)$.
Thus, every tree decomposition of $G'$ contains a bag inducing a subgraph with independence number at least $\alpha(G)$.
This shows that $\tin(G')\ge \alpha(G)$.
Therefore, equality must hold.
\end{proof}

Since computing the independence number of a graph is \NP-hard~\cite{MR0378476}, \cref{lem:alpha} implies the following.

\begin{theorem}\label{computing-tree-independence-number-is-hard}
Computing the tree-independence number of a given graph is \NP-hard.
\end{theorem}

For later use we also record the following consequence of \cref{lem:alpha}.

\begin{corollary}\label{tin-of-Knn}
For every positive integer $n$, we have $\tin(K_{n,n})  = n$.
\end{corollary}

Given any graph $G$, the trivial tree decomposition of $G$ has independence number equal to $\alpha(G)$.
This immediately implies the following.

\begin{observation}\label{tin is at most alpha}
For every graph $G$ we have $\tin(G)\le \alpha(G)$.
\end{observation}

As shown by \cref{tin-of-Knn}, this bound is sharp.
Another easily obtained upper bound on the tree-independence number of a graph is related to the treewidth of the graph.
This bound is sharp as well.

\begin{theorem}\label{tree-independence number bounded by treewidth}
For every graph $G$, $\tin(G)\le\tw(G)+1$, and this bound is sharp: for every integer $k\neq 2$, there exists a graph $G$ such that $\tin(G) = k$ and $\tw(G) = k-1$.
\end{theorem}

\begin{proof}
It follows directly from the definitions that the  independence number of any tree decomposition is at most its width plus one. Thus, taking $\mathcal{T}$ to be a tree decomposition of $G$ with minimum possible width, we obtain
$\tin(G)\le\alpha(\mathcal{T}) \le {\it width}(\mathcal{T})+1 = \tw(G)+1$.

For $k = 1$, the graph $K_1$ satisfies $\tin(K_1) = k$ and $\tw(K_1) = k-1$.

Fix $k\ge 3$, let $S = \{1,\ldots, k\}$, and let $G$ be the graph obtained from a complete graph with vertex set $S$ by replacing each of its edges $ij$ with $k$ paths of length two connecting $i$ and $j$. We claim that $\tw(G) = k-1$ and $\tin(G) = k$. Since $\tin(G)\le\tw(G)+1$, it suffices to show that $\tw(G) \le k-1$ and that $\tin(G) \ge k$.

The inequality $\tw(G) \le k-1$ follows from \cref{tw-via-triangulations} and the observation that the graph $G'$ obtained from $G$ by adding to it all edges between vertices in the set $S$ is a chordal graph with clique number $k$.

It remains to show that $\tin(G) \ge k$.
Consider an arbitrary tree decomposition $\mathcal{T} = (T,\{X_t\}_{t\in V(T)})$ of $G$.
We claim that there exists a bag of $\mathcal{T}$ having independence number at least $k$.
This is clearly the case if there exists a bag $X_t$ such that $S\subseteq X_t$.
Thus, we may assume that no bag of $\mathcal{T}$ contains $S$.
By \cref{conformality-of-bags}, there exist two distinct vertices $i,j\in S$ that are not contained in the same bag of $\mathcal{T}$.
Since $i$ and $j$ are not contained in a same bag of $\mathcal T$, the subtrees $T_i$ and $T_j$ are disjoint.
Note that, by construction of $G$, the vertices $i$ and $j$ have $k$ common neighbors.
Let $u \in N(i) \cap N(j)$ be such a vertex.
Then $u$ belongs to a bag in $\mathcal T$ containing $i$, and similarly for $j$.
In other words, $T_u$ intersects $T_i$ and $T_j$, which implies that $T_i \cup T_u \cup T_j$ is connected.
Let $P$ be the path in $T$ connecting $T_i$ and $T_j$.
Clearly, $P$ is a subgraph of $T_u$, and thus for any node $t \in P$, the bag $X_t$ of $\mathcal T$ contains $u$.
Since this holds for every common neighbor of $i$ and $j$, every such bag contains $N(i) \cap N(j)$, which is an independent set of size $k$.
Thus, as $\mathcal T$ can be any tree decomposition of $G$, we conclude that $\tin(G) \geq k$.
\end{proof}

Note that there are no graphs $G$ with $\tin(G) = \tw(G)+1 = 2$.
Any such graph would satisfy $\tw(G) = 1$ and would therefore be acyclic, and hence chordal.
\Cref{chordal} would then imply that $\tin(G)\le 1$, a contradiction.

\medskip
The treewidth is well-known to be monotone under minors, in the sense that $\tw(G)\le \tw(G')$ for any graph $G$ and any minor $G'$ of $G$.
This is not the case for the tree-independence number, as it can increase under edge deletions.
This fact can be easily observed, e.g., using the families of complete graphs and complete bipartite graphs (cf. \cref{chordal,tin-of-Knn}).
On the other hand, the tree-independence number cannot be increased by vertex deletions and edge contractions.

\begin{proposition}\label{lem:tree-indepencence number induced-minor}
Let $G$ be a graph and $G'$ an induced minor of $G$.
Then $\tin(G')\le \tin(G)$.
\end{proposition}

\begin{proof}
Let $\mathcal{T} = (T,\{X_{t}\}_{t \in V(T)})$ be an arbitrary tree decomposition of $G$.
First, we show that the deletion of a vertex does not increase the tree-independence number.
Let $v$ be a vertex of $G$.
Let $\mathcal{T}'$ be the tree decomposition obtained from $\mathcal{T}$ by removing $v$ from all of the bags that contain it.
Observe that $\mathcal T'$ is a tree decomposition of $G-v$.
Clearly, $\alpha(\mathcal{T}')\le \alpha(\mathcal{T})$, and hence $\tin(G-v)\le \tin(G)$.

Now, we show that the contraction of an edge does not increase the tree-independence number.
Let $e = uv$ be an edge of $G$ and $G/e$ denote the graph obtained from $G$ after contracting the edge $e$.
We denote by $w$ the vertex of $G/e$ that corresponds to the contracted edge.
We construct a tree decomposition $\mathcal{T}' = (T,\{X'_t\}_{t \in V(T)})$ (where $T$ is the same tree as in $\mathcal T$) of $G/e$ as follows.
For each node $t$ of $\mathcal T'$, we have two cases: if $X_{t}$ contains neither $u$ nor $v$, then we set $X'_t = X_t$; otherwise,
we set $X'_t = (X_t \setminus \{u,v\}) \cup \{w\}$.
We claim that $\mathcal T'$ is a tree decomposition of $G/e$ such that $\alpha(\mathcal T') \leq \alpha(\mathcal T)$.
First, observe that $\mathcal T'$ is a tree decomposition of $G/e$, as it satisfies all the defining conditions of a tree decomposition.
To verify that $\alpha(\mathcal T') \leq \alpha(\mathcal T)$, fix a bag $X'_t$ of $\mathcal T'$ and let $I\subseteq X'_t$ be an independent set in $G/e$.
If $w \notin I$, then $I$ also corresponds to an independent set in $G[X_t]$, and hence $|I| \leq \alpha(G[X_t])$.
On the other hand, if $w \in I$, then either $(I\setminus \{w\}) \cup \{u\})$ or $(I\setminus \{w\}) \cup \{v\})$ is an independent set in $G[X_t]$.
In particular, we again get that $|I| \leq \alpha(G[X_t])$.
Thus, we have in both cases that $|I| \leq \alpha(\mathcal{T})$.
It follows that $\alpha(\mathcal{T}')\le \alpha(\mathcal{T})$, and hence $\tin(G/e)\le \tin(G)$.

If $G'$ is an induced minor of $G$, then $G'$ can be obtained from $G$ by a sequence of vertex deletions and edge contractions, which implies that $\tin(G')\le \tin(G)$.
\end{proof}

We now show a result that we will use in the proof of \cref{sufficient conditions for bounded tree-alpha}, and which can find further applications when combined with Tarjan's decomposition of a graph along clique cutsets~\cite{MR798539}.
A \emph{cut-partition} in a graph $G$ is a triple $(A,B,C)$ such that $A\cup B\cup C = V(G)$, $A\neq \emptyset$, $B\neq \emptyset$, and $G$ has no edges with one endpoint in $A$ and one in $B$.
A \emph{clique cutset} in $G$ is a clique $C$ such that $G$ has a cut-partition $(A,B,C)$ for some $A$ and $B$.

\begin{proposition}\label{reduction-to-atoms-basic}
Let $C$ be a clique cutset in a graph $G$ and let $(A,B,C)$ be a cut-partition of $G$.
Let $G_A = G[A\cup C]$ and $G_B = G[B\cup C]$, and let $\mathcal{T}_A$ and $\mathcal{T}_B$ be tree decompositions of $G_A$ and $G_B$, respectively.
Then we can compute in time $\mathcal{O}(|\mathcal{T}_A|+|\mathcal{T}_B|)$ a tree decomposition $\mathcal{T}$ of $G$ such that $\alpha(\mathcal{T}) = \max\{\alpha(\mathcal{T}_A),\alpha(\mathcal{T}_B)\}$.
\end{proposition}

\begin{proof}
Since $C$ is a clique in $G_A$ and in $G_B$, by \cref{clique in bag} there exists a bag $X_A$ of ${\mathcal T}_A$ such that $C\subseteq X_A$, and a bag $X_B$ in ${\mathcal T}_B$ such that $C\subseteq X_B$.
Take the disjoint union of the tree $T_A$ of ${\mathcal T}_A$ with the tree $T_B$ of ${\mathcal T}_B$, and add an edge connecting the nodes corresponding to $X_A$ and $X_B$.
This results in a tree decomposition $\mathcal{T} = (T, \{(X_t,U_t)\}_{t\in V(T)})$ of $G$.
Indeed, since every vertex of $G$ is a vertex of $G_A$ or $G_B$, every vertex of $G$ is in at least one bag of $\mathcal{T}$.
A similar argument shows that for every edge of $G$, both endpoints belong to a common bag.
Finally, let us verify that for every vertex $u\in V(G)$, the subgraph of $T$ induced by the set of bags containing $u$ is connected.
If $u\in A\cup B$, then this follows from the corresponding properties of ${\mathcal T}_A$ and ${\mathcal T}_B$.
Suppose now that $u\in C$. Then the subgraph $T_A(u)$ of $T_A$ induced by the set of bags of ${\mathcal T}_A$ containing $u$ is connected and contains the node corresponding to $X_A$.
Similarly, the subgraph $T_B(u)$ of $T_B$ induced by the set of bags of ${\mathcal T}_B$ containing $u$ is connected and contains the node corresponding to $X_B$.
Thus, the subgraph $T(u)$ of $T$ induced by the set of bags of $\mathcal{T}$ containing $u$ is isomorphic to the graph obtained from the disjoint union of $T_A(u)$ and $T_B(u)$ together with the edge between $X_A$ and $X_B$, and hence $T(u)$ is connected.
By construction, the independence number of $\mathcal T$ satisfies $\alpha(\mathcal{T}) = \max\{\alpha(\mathcal{T}_A),\alpha(\mathcal{T}_B)\}$.
Finding a bag $X_A$ such that $C\subseteq X_A$ can be done in time $\mathcal{O}(|\mathcal{T}_A|)$, and similarly for $X_B$.
Thus, the time complexity of the above procedure is
$\mathcal{O}(|\mathcal{T}_A|+|\mathcal{T}_B|)$.
\end{proof}

\begin{corollary}\label{reduction-to-atoms-tree-alpha}
Let $C$ be a clique cutset in a graph $G$, let $(A,B,C)$ be a cut-partition of $G$, and let $G_A = G[A\cup C]$ and $G_B = G[B\cup C]$.
Then $\tin(G) = \max\{\tin(G_A),\tin(G_B)\}$.
\end{corollary}

\begin{proof}
To justify that $\tin(G) = \max\{\tin(G_A),\tin(G_B)\}$, note first that $G_A$ and $G_B$ are induced subgraphs of $G$ and hence $\max\{\tin(G_A),\tin(G_B)\}\le \tin(G)$ by \cref{lem:tree-indepencence number induced-minor}.
For the converse inequality, let $\mathcal{T}_A$ and $\mathcal{T}_B$ be tree decompositions of $G_A$ and $G_B$, respectively, such that $\alpha(\mathcal{T}_A) = \tin(G_A)$ and $\alpha(\mathcal{T}_B) = \tin(G_B)$.
\Cref{reduction-to-atoms-basic} implies the existence of a tree decomposition $\mathcal{T}$ of $G$ such that $ \alpha(\mathcal{T}) = \max\{\alpha(\mathcal{T}_A),\alpha(\mathcal{T}_B)\}$.
It follows that $\tin(G)\le \alpha(\mathcal{T}) = \max\{\alpha(\mathcal{T}_A),\alpha(\mathcal{T}_B)\} = \max\{\tin(G_A),\tin(G_B)\}$.
\end{proof}

We conclude this section by summarizing some sufficient conditions for a graph class to have bounded tree-independence number.
To state the result we need to introduce names for two further graph invariants.
Given a graph $H$ and connected subgraphs $H_1,\ldots, H_p$ of $H$, the \emph{intersection graph} of $H_1,\ldots, H_p$ is the graph with vertex set $\{1,\ldots, p\}$ in which two distinct vertices $i$ and $j$ are adjacent if and only if $H_i$ and $H_j$ have a vertex in common.
We say that a graph $G$ is \emph{representable} in a graph $H$ if there exist connected subgraphs $H_1,\ldots, H_p$ of $H$ such that $G$ is isomorphic to their intersection graph.
Note that every graph $G$ is representable in the graph $H$ obtained by subdividing each edge of $G$ exactly once: each vertex $v$ of $G$ corresponds to the subgraph of $H$ induced by the closed neighborhood of $v$ in $H$.
In particular, this means that the following graph invariant is well-defined.
The \emph{intersection treewidth} of a graph $G$, denoted by $\itw(G)$, is the smallest nonnegative integer $k$ such that there exists a graph $H$ with treewidth at most $k$ such that $G$ is representable in $H$.
Note that if $G$ is a chordal graph, then $\itw(G)\le 1$, and if $G$ is a circular-arc graph, then $\itw(G)\le 2$.

Furthermore, we define the \emph{maximum minimal separator size} of a graph $G$, denoted by $\mms(G)$, as the smallest nonnegative integer $k$ such that every minimal separator of $G$ has cardinality at most $k$.
If $G$ is a complete graph, then $G$ does not have any minimal separators.
We extend the definition to this case by setting $\mms(G) = 0$ if $G$ is complete.

\begin{theorem}\label{sufficient conditions for bounded tree-alpha}
Let $\G$ be a graph class that has either bounded treewidth, bounded independence number, bounded intersection treewidth, or bounded maximum minimal separator size.
Then $\G$ has bounded tree-independence number.
\end{theorem}

\begin{proof}
Fix a positive integer $k$, a graph invariant $\rho\in \{\tw,\alpha,\itw,\mms\}$ and a graph $G$ such that $\rho(G) \le k$.
It suffices to show that $\tin(G)\le f(k)$ for some function $f:\mathbb{Z}_+\to \mathbb{Z}_+$.
For $\rho = \tw$, \cref{tree-independence number bounded by treewidth} applies and we can take $f(k) = k+1$.
For $\rho = \alpha$, \cref{tin is at most alpha} applies and we can take $f$ to be the identity function.
Assume now that $\rho = \itw$, that is, $\itw(G) \le k$.
As shown by Bodlaender, Gustedt, and Telle in the proof of~\cite[Lemma~2.4]{MR1642971}, $G$ has a tree decomposition $\mathcal{T}$ in which each bag is the union of at most $k+1$ cliques in $G$.
This implies that $\mathcal{T}$ has independence number at most $k+1$ and hence $\tin(G)\le k+1$.
Finally, let $\rho = \mms$, that is, $\mms(G)\le k$.
We claim that $\tin(G)\le 2k-1$.
It follows from a result by Skodinis \cite{MR1852483} that $G$ is either a complete graph, or $G$ has a clique cutset, or $\tw(G)\le 2k-1$.
If $G$ is a complete graph, then $\tin(G)\le 1$.
If $G$ has treewidth at most $2k-1$, then $\tin(G)\le 2k-1$.
Assume now that $G$ has a clique cutset.
For every induced subgraph $G'$ of $G$ and every minimal separator $S'$ in $G'$ there exists a minimal separator $S$ in $G$ such that $S'\subseteq S$.
Indeed, if $S'$ is a minimal $u{,}v$-separator in $G'$, then any minimal $u{,}v$-separator in $G$ contained in the $u{,}v$-separator $S'\cup (V(G)\setminus V(G'))$ in $G$ satisfies $S'\subseteq S$.
Therefore, every induced subgraph $G'$ of $G$ has $\mms(G')\le \mms(G)\le k$, and we can use induction on $|V(G)|$ and \cref{reduction-to-atoms-basic} to infer that $\tin(G)\le 2k-1$.
\end{proof}

In the third paper of the series~\cite{dallard2022secondpaper}, we derive the fact the class of graphs of bounded maximum minimal separator size has bounded tree-independence number also as a consequence of a more general result, namely, that boundedness of the tree-independence number is achieved whenever the independence number of subgraphs induced by minimal separators is bounded.

\section{A refinement of the tree-independence number}\label{sec:ell-refined}

As explained in the introduction, one of the main motivations for introducing the tree-independence number of a graph is algorithmic: tree decompositions with small independence number allow for the development of efficient dynamic programming algorithms for the \textsc{Max Weight Independent Set} and related problems (see \cref{sec:MWIS,sec:implications}).
For the \textsc{Max Weight Independent Set} problem, given a graph $G$ equipped with a tree decomposition $\mathcal{T}$ with independence number at most $k$, the algorithm takes time $\mathcal{O}(|V(G)|^k)$ for each bag of $\mathcal{T}$, in order to consider all subsets of at most $k$ vertices from the bag.
In some cases, this approach can be improved.
Suppose that in each bag $X_t$ we can identify a set $U_t\subseteq X_t$ of at most a constant number $\ell$ of vertices such that the independence number of the subgraph of $G$ induced by the remaining vertices, $X_t\setminus U_t$, is at most $k$.
Then, even though we can only bound the independence number of the subgraph of $G$ induced by $X_t$ by $k+\ell$, we do not need to examine all $\mathcal{O}(|X_t|^{k+\ell}) = \mathcal{O}(|V(G)|^{k+\ell})$ subsets of the bag of size at most $k+\ell$, but only $\mathcal{O}(2^{\ell}\cdot |X_t\setminus U_t|^{k})$ of them, since every independent set of $G[X_t]$ consists of a subset of $U_t$ along with at most $k$ of the remaining vertices.
For constant values of $\ell$, the algorithm will thus take time $\mathcal{O}(|V(G)|^k)$ instead of $\mathcal{O}(|V(G)|^{k+\ell})$ for each bag, which may lead to a significantly improved running time.
This idea is formalized in the concepts of $\ell$-refined tree decompositions and their residual independence numbers.

\begin{definition}\label{definition-ell-refined}
Given a nonnegative integer $\ell$, an \emph{$\ell$-refined tree decomposition} of a graph $G$ is a pair $\widehat{\mathcal T} = (T, \{(\Bag_t,U_t)_{t \in V(T)}\})$ such that $\mathcal T = (T,  \{\Bag_t\}_{t\in V(T)})$ is a tree decomposition of $G$, and for every $t \in V(T)$ we have $U_t \subseteq X_t$ and $|U_t| \leq \ell$.
We will refer to $\mathcal{T}$ as the \emph{underlying tree decomposition} of~$\widehat{\mathcal T}$.
The \emph{width} of $\widehat{\mathcal T}$, denoted by ${\it width}(\widehat{\mathcal T})$, is defined as the width of $\mathcal{T}$.
More generally, any concept defined for tree decompositions can be naturally extended to $\ell$-refined tree decompositions, simply by considering it on the underlying tree decomposition.
\end{definition}

\begin{definition}\label{definition-ell-refined-tree-alpha}
Given a nonnegative integer $\ell$, the \emph{residual independence number} of an $\ell$-refined tree decomposition $\widehat{\mathcal T}$ of a graph $G$ is defined as $\max_{t \in V(T)} \alpha(G[X_t \setminus U_t])$ and denoted by $\widehat{\alpha}(\widehat{\mathcal T})$.
The \emph{$\ell$-refined tree-independence number} of a graph $G$ is defined as the minimum residual independence number of an $\ell$-refined tree decomposition of $G$, and denoted by $\ell$\textnormal{-}$\tin(G)$.
\end{definition}

We showcase the applicability of these notions in the third paper of the series~\cite{dallard2022secondpaper}, by showing that such tree decompositions exist, and can be computed efficiently, for any graph excluding either a $W_4$ or a $K_5^-$ as an induced minor.

The notion of an $\ell$-refined tree decomposition with bounded residual independence number generalizes the notion of $(k,\ell)$-semi clique tree decompositions studied by Jacob et al.~in~\cite{MR4189425}.
Given two nonnegative integers $k$ and $\ell$, a \emph{$(k,\ell)$-semi clique tree decomposition} of $G$ is a tree decomposition $\mathcal T = (T,  \{\Bag_t\}_{t\in V(T)})$ of $G$ such that in each bag $X_t$ there exists a subset $U_t\subseteq X_t$ and $k$ (possibly empty) cliques $C_1,\ldots, C_k$ in $G$ such that $|U_t|\le \ell$ and $\cup_{i = 1}^k C_i = X_t\setminus U_t$.
Since the independence number of any graph that can be covered by $k$ cliques is at most $k$, the $\ell$-refined tree-independence number of any graph $G$ admitting a $(k,\ell)$-semi clique tree decomposition is at most $k$.

Let us say that a graph $G=(V,E)$ is \emph{$\ell$-vertex-almost chordal} if $G$ contains a set $S\subseteq V$ of at most $\ell$ vertices such that $G-S$ is a chordal graph.
Note that it follows from \cref{chordal} that for any nonnegative integer $\ell$, any $\ell$-vertex-almost chordal graph admits an $\ell$-refined tree decomposition with residual independence number at most $1$ (or, equivalently, a $(1,\ell)$-semi clique tree decomposition).
Jacob et al.~showed in~\cite{MR4189425} that for any fixed $\ell$ and any $\ell$-vertex-almost chordal graph $G$ a $(4,7\ell+5)$-semi clique tree decomposition of $G$ can be computed in polynomial time.

A similar but different distance to chordality can be obtained by considering edge additions.
Let us say that a graph $G=(V,E)$ is \emph{$\ell$-edge-almost chordal} if $G$ can be turned into a chordal graph by adding at most $\ell$ edges.
In this respect, Fomin and Golovach defined the concept of an \emph{$\ell$-edge-almost chordal} tree decomposition of the graph as a tree decomposition where each bag can be made a clique by adding at most $\ell$ edges~\cite{MR4276552} (they called such tree decompositions \emph{$\ell$-almost chordal} tree decompositions). 
It follows from the definitions that every $\ell$-edge-almost chordal graph admits an $\ell$-edge-almost chordal tree decomposition, and that the $\ell$-refined tree-independence number of any graph $G$ admitting an $\ell$-edge-almost chordal tree decomposition is at most $1$.
Given an $\ell$-almost chordal tree decomposition, Fomin and Golovach construct dynamic programming algorithms that are subexponential in $\ell$ for various problems. 
In particular, they showed that for any fixed $\ell$ and any $\ell$-edge-almost chordal graph $G$, an $\ell$-almost chordal tree decomposition of $G$ can be computed in polynomial time.

\medskip
Note that each $\ell$-refined tree decomposition $\widehat{\mathcal T} = (T, \{(\Bag_t,U_t)_{t \in V(T)}\})$ of a graph $G$ is also an \hbox{$(\ell+1)$-refined} tree decomposition, and hence $(\ell+1)\textnormal{-}\tin(G)\le \ell\textnormal{-}\tin(G)$ for all $\ell\ge 0$ and all graphs $G$.
Furthermore, for each $\ell\ge 0$, every tree decomposition $\mathcal T = (T,  \{\Bag_t\}_{t\in V(T)})$ of a graph $G$ naturally corresponds to an $\ell$-refined tree decomposition $\widehat{\mathcal T} = (T, \{(\Bag_t,U_t)_{t \in V(T)}\})$ of $G$ such that $\widehat{\alpha}(\widehat{\mathcal T}) = \alpha({\mathcal T})$, obtained by setting $U_t = \emptyset$ for all $t\in V(T)$.
On the other hand, for any $\ell$-refined tree decomposition $\widehat{\mathcal{T}}$ of a graph $G$, the independence number of the underlying tree decomposition of $\widehat{\mathcal{T}}$ is at most $\widehat{\alpha}(\widehat{\mathcal{T}}) + \ell$.
We thus have:

\begin{observation}\label{observation-ell-tree-alpha}
For every graph $G$ and every integer $\ell\ge 0$, we have
\[\ell\textnormal{-}\tin(G)\le \tin(G)\le \ell\textnormal{-}\tin(G) + \ell\,.\]
In particular, equalities hold when $\ell = 0$, that is, $0\textnormal{-}\tin(G) = \tin(G)$.
\end{observation}

Using $\ell$-refined tree decompositions, we can refine \cref{bounded tin implies bounded tw-omega} as follows.

\begin{lemma}\label{bounded tin implies bounded tw-omega-refined}
For every two non-negative integers $k$ and $\ell$, the class of graphs with $\ell$-refined tree-independence number at most $k$ is $(\tw,\omega)$-bounded, with a binding function $f(p) = R(p+1,k+1)+\ell-2$, which is upper-bounded by a polynomial in $p$ of degree $k$.
\end{lemma}

\begin{proof}
Recall that $R(p,q)\le \binom{p+q-2}{q-1}$ for all positive integers $p$ and $q$, which, for fixed $q$, is a polynomial in $p$ of degree $q-1$.

    Let us now fix $p \in \mathbb{Z}_+$ and let $G$ be a graph such that $\omega(G) = p$ and $\ell$-$\tin(G) \leq k$.
    Fix an $\ell$-refined tree decomposition $\widehat{\mathcal T} = (T,\{(X_t,U_t)_{t\in V(T)}\})$ of $G$ with residual independence number at most $k$.
    For every node $t\in V(T)$, the set $X_t\setminus U_t$ induces a subgraph of $G$ with independence number at most $k$ and clique number at most $p$.
    Thus, for every bag $X_t$ of $\widehat{\mathcal T}$, Ramsey's theorem implies that $|X_t| \leq |U_t|+ R(p+1,k+1)-1 \le R(p+1,k+1)+\ell-1$.
    It follows that $\tw(G) \leq R(p+1,k+1)+\ell-2$, as claimed.
\end{proof}

By~\cref{chordal}, a graph $G$ has a $0$-refined tree decomposition with (residual) independence number at most $1$ if and only if $G$ is chordal.
In this case \cref{tw-via-triangulations} implies that $\tw(G) = \omega(G)-1$.
This result generalizes to graphs admitting an $\ell$-refined tree decomposition with residual independence number at most $1$, for a nonnegative integer $\ell$.
Such a tree decomposition yields an additive approximation for treewidth to within $\ell$.

\begin{proposition}\label{tw approx for ell-refined}
Let $\ell \in \mathbb{Z}_+$ and let $G$ be a graph admitting an $\ell$-refined tree decomposition $\widehat{\mathcal T}$ with residual independence number at most $1$.
Then \hbox{${\it width}(\widehat{\mathcal T})\le \tw(G) +\ell$} and $\tw(G) \le \omega(G)-1+\ell$.
\end{proposition}

\begin{proof}
Let us write $\widehat{\mathcal T} = (T, \{(\Bag_t,U_t)_{t \in V(T)}\})$ and let $\Bag_t$ be a largest bag of the underlying tree decomposition.
Since $X_t\setminus U_t$ induces a subgraph with independence number at most one, it is a clique in $G$.
We thus obtain, using also \cref{clique-tw-bound}, that
\[{\it width}(\widehat{\mathcal T})= |X_t|-1\le |U_t|+|X_t\setminus U_t|-1\le \ell+\omega(G)-1\le \ell+\tw(G)\,.\qedhere\]
\end{proof}

A result analogous to \cref{reduction-to-atoms-basic} (and with the same proof) also holds for $\ell$-refined tree decompositions.
This result is used in the third paper of the series~\cite{dallard2022secondpaper}.

\begin{proposition}\label{reduction-to-atoms}
Let $C$ be a clique cutset in a graph $G$ and let $(A,B,C)$ be a cut-partition of $G$.
Let $G_A = G[A\cup C]$ and $G_B = G[B\cup C]$, and let $\widehat{\mathcal{T}}_A$ and $\widehat{\mathcal{T}}_B$ be $\ell$-refined tree decompositions of $G_A$ and $G_B$, respectively.
Then we can compute in time $\mathcal{O}(|\widehat{\mathcal{T}}_A|+|\widehat{\mathcal{T}}_B|)$ an  $\ell$-refined tree decomposition $\widehat{\mathcal{T}}$ of $G$ such that $\widehat{\alpha}(\widehat{\mathcal{T}}) = \max\{\widehat{\alpha}(\widehat{\mathcal{T}}_A),\widehat{\alpha}(\widehat{\mathcal{T}}_B)\}$.
\end{proposition}

\section{Weighted independent sets}\label{sec:MWIS}

In this section, we prove that every graph class with bounded and efficiently witnessed tree-independence number admits a polynomial-time algorithm for the \textsc{Max Weight Independent Set} problem.

A tree decomposition $\mathcal T = (T, \{\Bag_t\}_{t\in V(T)})$ of a graph $G$ is said to be \emph{rooted} if we distinguish one node $r$ of $T$ which we take as the root of $T$.
This introduces natural parent-child and ancestor-descendant relations in the tree $T$.
A \emph{leaf} of $T$ is a node with no children.
Following~\cite{MR3380745}, we say that a tree decomposition $(T, \{\Bag_t\}_{t\in V(T)})$ is \emph{nice} if it is rooted and the following conditions are satisfied:
\begin{enumerate}[label=(\alph*)]
\item If $t\in V(T)$ is the root or a leaf of $T$, then $\Bag_t = \emptyset$;
\item Every non-leaf node $t$ of $T$ is one of the following three types:
  \begin{itemize}
\item {\bf Introduce node:} a node $t$ with exactly one child $t'$ such that $\Bag_t = \Bag_{t'}\cup\{v\}$ for some vertex $v\in V(G)\setminus \Bag_{t'}$;
\item {\bf Forget node:} a node $t$ with exactly one child $t'$ such that $\Bag_t = \Bag_{t'}\setminus\{v\}$ for some vertex $v\in \Bag_{t'}$;
\item {\bf Join node:} a node $t$ with exactly two children $t_1$ and $t_2$ such that $\Bag_t = \Bag_{t_1} = \Bag_{t_2}$.
\end{itemize}
\end{enumerate}
We assume that introduce and forget nodes are also labeled with the vertex $v$ which is introduced or forgotten.

It is well known that, given a graph $G$ and a tree decomposition $\mathcal T$ of $G$ with width at most $k$, one can compute a nice tree decomposition of $G$ with width at most $k$ in polynomial time (see, e.g.,~\cite{MR3380745}).
The standard approach for doing so can in fact be used to prove the following lemma.

\begin{sloppypar}
\begin{lemma}\label{lem:nice}
Given an $\ell$-refined tree decomposition $\widehat{\mathcal T} = (T, \{(\Bag_t,U_t)\}_{t\in V(T)})$ with width $k$ of a graph $G$, one can compute in time $\mathcal{O}(k^2\cdot|V(T)|)$ a nice $\ell$-refined tree decomposition
\hbox{$\widehat{\mathcal T}' = (T', \{(\Bag_{t'},U_{t'})\}_{t'\in V(T')})$} of $G$ that has at most $\mathcal{O}(k\cdot|V(T)|)$ nodes and such that for every node $t'\in V(T')$ there exists a node  $t\in V(T)$ such that $X_{t'} \subseteq X_t$ and $U_{t'} = U_t\cap X_{t'}$.
In particular, $\widehat{\alpha}(\widehat{\mathcal T}')\le \widehat{\alpha}(\widehat{\mathcal T})$.
\end{lemma}
\end{sloppypar}

\begin{proof}
Consider the following algorithm.

\begin{enumerate}
    \item We traverse the tree $T$ and check for every two adjacent nodes $t$ and $t'$ in $T$ if the bags $X_t$ and $X_{t'}$ are comparable (i.e., if $X_t \subseteq X_{t'}$ or $X_{t'} \subseteq X_{t}$).
    If, say, $X_t \subseteq X_{t'}$, then we contract the edge $tt'$ and label the resulting node with the pair $(X_{t'},U_{t'})$.
    Hence, we now assume that no two adjacent nodes of $T$ have comparable bags.

    \item We choose a node of $T$ with degree at most one as its root $r$ and compute the corresponding parent-child relationship in $T$.

    \item We assure that each node has at most two children, by replacing each node $t$ of $T$ with $d\ge 3$ children $c_1,\ldots, c_d$ with a path consisting of $d$ nodes $t_1,\ldots, t_{d}$, each associated with the same pair $(X_t,U_t)$, making $t_1$ a child of the parent of $t$, and for all $j\in \{1,\ldots, d\}$, making $c_j$ a child of $t_j$ (note that if $j<d$, then $t_{j+1}$ is also a child of $t_j$).

    \item For every node $t$ of $T$ with exactly two children $t_1$ and $t_2$, we label $t$ as a join node.
    For $i \in \{1,2\}$, if $X_t \neq X_{t_{i}}$, then we subdivide the edge $tt_i$ and associate the pair $(X_t,U_t)$ to the new node.

    \item For each leaf node $t$ of $T$, we add a new node $t'$ associated with the pair $(\emptyset,\emptyset)$ and make $t'$ a child of $t$.

    \item We add a new node $r'$ associated with the pair $(\emptyset,\emptyset)$ and make $r$ a child of $r'$ (that is, $r'$ becomes the new root).

    \item For every node $t$ of $T$ that is not already labeled as a join node there is a unique child $t'$ of~$t$.
    We replace the edge $t't$ with a path $(t' = t_0,\ldots, t_p, t_{p+1},\ldots, t_{p+q} = t)$ of length $p+q$ where $p = |X_{t'}\setminus X_t|$ and $q = |X_{t}\setminus X_{t'}|$.
    For all $i\in \{1,\ldots, p\}$, the node $t_i$ is a forget node labeled with $(X_{t_i},U_{t_i})$ where $X_{t_i}$ forgets one vertex from $X_{t_{i-1}}\setminus X_t$ and $U_{t_i} = U_{t'}\cap X_{t_i}$.
    Similarly, for all $j\in \{1,\ldots, q-1\}$, the node $t_{p+j}$ is an introduce node labeled with $(X_{t_{p+j}},U_{t_{p+j}})$ where $X_{t_{p+j}}$ introduces one vertex from $X_t\setminus X_{t_{p+j-1}}$, and $U_{t_{p+j}} = U_{t}\cap X_{t_{p+j}}$.
    The last node of the path is the node $t$, which is an introduce node.
\end{enumerate}
Note that each of the above steps modifies an $\ell$-refined tree decomposition into another one.
Let us denote by $\widehat{\mathcal{T}}' =  (T', \{(\Bag_{t'},U_{t'})\}_{t'\in V(T')})$ the final $\ell$-refined tree decomposition.
By construction, every non-leaf node of $T'$ has a unique label (an introduce node, a forget node, or a join node) and $\widehat{\mathcal T}'$ is a nice $\ell$-refined tree decomposition of $G$ such that for every node $t'$ of $T'$, there exists a node $t$ in $T$ such that $X_{t'} \subseteq X_t$ and $U_{t'} = U_t\cap X_{t'}$.
Consider a node $t'$ of $T'$ such that $\widehat{\alpha}(\widehat{\mathcal{T}}') = \alpha(G[X_{t'}\setminus U_{t'}])$ and let $t$ be a node of $T$ such that $X_{t'} \subseteq X_t$ and $U_{t'} = U_t\cap X_{t'}$.
Then, $X_{t'}\setminus U_{t'} = X_{t'}\setminus (U_t\cap X_{t'}) = X_{t'}\setminus U_t\subseteq X_t\setminus U_t$ and hence
$\widehat{\alpha}(\widehat{\mathcal{T}}') = \alpha(G[X_{t'}\setminus U_{t'}])\le \alpha(G[X_{t}\setminus U_{t}])\le \widehat{\alpha}(\widehat{\mathcal{T}})$.

We now reason about the complexity of obtaining $\widehat{\mathcal T}'$ by considering the complexity of each aforementioned step.
Step 1 takes $\mathcal{O}(k^2)$ time for every edge of $T$, and thus $\mathcal{O}(k^2 \cdot |V(T)|)$ overall; the resulting tree $T'$ has $\mathcal{O}(|V(T)|)$ nodes.
Step 2 can be done in time $\mathcal{O}(|V(T)|)$.
Step 3 can be done in time $\mathcal{O}(|V(T)|)$ and results in a tree $T'$ with $\mathcal{O}(|V(T)|)$ nodes.
Step 4 can be done in time $\mathcal{O}(k^2 \cdot |V(T)|)$ and Step 5 in time $\mathcal{O}(|V(T)|)$; both steps yield a tree with $\mathcal{O}(|V(T)|)$ nodes.
Step 6 can be done in constant time.
Finally, Step 7 can be done in time $\mathcal{O}(k^2\cdot |V(T)|)$ and results in an $\ell$-refined tree decomposition with $\mathcal{O}(k \cdot |V(T)|)$ nodes.
\end{proof}

In the next theorem, we adapt the standard dynamic programming approach (see, e.g.,~\cite{MR3380745}) for solving the \textsc{Max Weight Independent Set} problem in graphs of bounded treewidth to graphs of bounded $\ell$-refined tree independence number, for some integer $\ell \geq 0$.
In fact, the algorithm runs in polynomial time even if $\ell = \mathcal{O}(\log |V(G)|)$, which is the case, for example, for graphs with logarithmic treewidth (see, e.g.,~\cite{MR4480493,MR4538068}).

\begin{theorem}\label{thm:bounded-ell-refined-tree-independence-number}
For every integer $k\ge 1$, \textsc{Max Weight Independent Set} is solvable in time $\mathcal{O}(2^\ell\cdot |V(G)|^{k+1}\cdot|V(T)|)$ if the input vertex-weighted graph $G$ is given with an $\ell$-refined tree decomposition \hbox{$\widehat{\mathcal{T}} = (T, \{(\Bag_t,U_t)\}_{t\in V(T)})$} with residual independence number at most $k$.
\end{theorem}

\begin{proof}
Let $n = |V(G)|$ and $w:V(G)\to \mathbb{Q}_+$ be the weight function.
We first apply \cref{lem:nice} and compute in time $\mathcal{O}(n^2\cdot|V(T)|)$
a nice $\ell$-refined tree decomposition $\widehat{\mathcal{T}}' = (T', \{(\Bag_t,U_t)\}_{t\in V(T')})$ of $G$ with $\mathcal{O}(n\cdot |V(T)|)$ bags and such that $\widehat{\alpha}(\widehat{\mathcal T}')\le \widehat{\alpha}(\widehat{\mathcal T})\le k$.

Recall that, by definition, $\widehat{\mathcal{T}}'$ is rooted at some node $r$ of $T'$.
For every node $t\in V(T')$, we denote by $V_t$ the union of all bags $\Bag_{t'}$ such that $t'\in V(T')$ is a (not necessarily proper) descendant of $t$ in $T'$.

For each node $t\in V(T')$, we compute the family ${\mathcal S}_t$ of all sets $S\subseteq X_t$ that are independent in $G$.
Note that each set $S\in {\mathcal S}_t$ is the disjoint union of sets $S_1$ and $S_2$ where $S_1 = S\cap U_t$ and $S_2 =  S \cap (X_t \setminus U_t)$.
Since $\widehat{\mathcal{T}}'$ is an $\ell$-refined tree decomposition with residual independence number at most~$k$, we have that $|S_1|\le |U_t|\le \ell$ and $|S_2|\le \alpha(G[X_t\setminus U_t])\le k$.
It follows that the family ${\mathcal S}_t$ can be computed in time $\mathcal{O}(2^{|U_t|} \cdot |X_t\setminus U_t|^k)=\mathcal{O}(2^\ell\cdot n^k)$ by enumerating all $\mathcal{O}(2^\ell)$ candidate sets for $S_1$, all $\mathcal{O}(n^k)$ candidate sets for $S_2$, and verifying if the union $S_1\cup S_2$ is independent in $G$.\footnote{We can check in constant time if a set $S$ with $|S|\le k+\ell$ is independent in $G$, since $k+\ell$ is constant and we assume that $G$ is represented with an adjacency matrix---if not, we can first compute such a representation from the adjacency lists in time $\mathcal{O}(n^2)$.}
We traverse the tree $T'$ bottom-up and use a dynamic programming approach to compute, for every node $t\in V(T')$ and every set $S\in {\mathcal S}_t$, the value of $c[t,S]$, defined as the maximum weight of an independent set $I$ in the graph $G[V_t]$ such that $I\cap X_t = S$.

Since $\widehat{\mathcal{T}}'$ is nice, we have $\Bag_r = \emptyset$; in particular, the only independent set $S$ with $S\subseteq X_r$ is the empty set.
Furthermore, $V_r = V(G)$, and hence $c[r,\emptyset]$ corresponds to the maximum weight of an independent set in $G$, which is what we want to compute.

We consider various cases depending on the type of a node $t\in V(T')$. %
For each type we give a formula for computing the value $c[t,S]$ for all $S\in {\mathcal S}_t$ from the already computed values of $c[t',S']$ where $t'$ is a child of $t$ in $T'$ and $S'\in {\mathcal S}_{t'}$.

\smallskip\noindent\textbf{Leaf node.} By the definition of a nice tree decomposition it follows that $\Bag_t=\emptyset$.
Thus, we have ${\mathcal S}_t = \{\emptyset\}$, $V_t = \emptyset$, and $c[t,\emptyset] = 0$.

\smallskip\noindent\textbf{Introduce node.}
By definition, $t$ has exactly one child $t'$ and $\Bag_t = \Bag_{t'}\cup \{v\}$ holds for some vertex $v\in V(G)\setminus X_{t'}$. For an arbitrary set $S\in \mathcal{S}_t$, we have
\[
    c[t,S] = \left\{\begin{array}{ll}
    c[t',S] & \text{if~} v \notin S\,,\\
    c[t', S \setminus \{v\}] + w(v)& \text{otherwise.}
    \end{array}\right.
\]

\smallskip\noindent\textbf{Forget node.}
By definition, $t$ has exactly one child $t'$ in $T'$ and $\Bag_{t}=\Bag_{t'}\setminus  \{v\}$ holds for some vertex $v\in \Bag_{t'}$.
Note that for a set $S\in \mathcal{S}_t$, the set $S\cup \{v\}$ belongs to $\mathcal{S}_{t'}$ if and only if it is independent in $G$, that is, if no vertex in $S$ is adjacent to $v$.
For an arbitrary set $S\in \mathcal{S}_t$, we have
\[
    c[t,S] = \left\{\begin{array}{ll}
    \max\{c[t',S], c[t',S \cup \{v\}]\}& \text{if no vertex in $S$ is adjacent to $v$,}\\
    c[t',S] & \text{otherwise.}
    \end{array}\right.
\]

\smallskip\noindent\textbf{Join node.}
By definition, $t$ has exactly two children $t_1$ and $t_2$ in $T'$ and it holds that $\Bag_t = \Bag_{t_1} = \Bag_{t_2}$. For an arbitrary set $S\in \mathcal{S}_t$, we have
\[
    c[t,S] = c[t_1,S] + c[t_2,S] - w(S)\,.
\]

The only way in which our algorithm differs from the standard one (see~\cite{MR3380745}) is that we compute $c[t,S]$ only for sets $S$ in the family ${\mathcal S}_t$, and not for all subsets of the bag $X_t$.
We therefore omit the description of the recurrence relations leading to the dynamic programming algorithm and the proof of correctness.

It remains to estimate the time complexity.
We need time $\mathcal{O}(n^2\cdot|V(T)|)$ to compute $\widehat{\mathcal{T}}'$.
At each of the $\mathcal{O}(n\cdot |V(T)|)$ nodes $t\in V(T')$, we perform a constant-time computation for each set $S\in \mathcal{S}_t$, resulting in an overall time complexity of $\mathcal{O}(2^\ell\cdot n^k)$ per node.
Thus, the total time complexity of the algorithm is
$\mathcal{O}(n^2\cdot|V(T)|) + \mathcal{O}(n\cdot |V(T)|\cdot 2^\ell \cdot n^k) = \mathcal{O}(2^\ell \cdot n^{k+1}\cdot |V(T)|)$, as claimed.
\end{proof}

Since every tree decomposition $\mathcal{T}$ naturally corresponds to $0$-refined tree decomposition $\widehat{\mathcal{T}}$ with
the same independence number (more precisely, with $\widehat{\alpha}(\widehat{\mathcal{T}}) = \alpha(\widehat{\mathcal{T}}) = \alpha(\mathcal{T})$), \cref{thm:bounded-ell-refined-tree-independence-number} implies the following.

\begin{corollary}\label{thm:bounded-tree-independence-number}
For every $k\ge 1$, \textsc{Max Weight Independent Set} is solvable in time $\mathcal{O}(|V(G)|^{k+1}\cdot|V(T)|)$ if the input vertex-weighted graph $G$ is given with a tree decomposition \hbox{$\mathcal{T} = (T, \{\Bag_t\}_{t\in V(T)})$} with independence number at most $k$.
\end{corollary}

\Cref{thm:bounded-tree-independence-number} generalizes a result of Krause~\cite{DBLP:conf/cc/Krause13} stating that the \textsc{Max Independent Set} problem is solvable in polynomial time on intersection graphs of connected subgraphs of graphs of bounded treewidth (assuming that the input graph is given with an intersection model).
Recall that every graph class with bounded intersection treewidth has bounded tree-independence number (cf. \cref{sufficient conditions for bounded tree-alpha}).
More precisely, consider a fixed nonnegative integer $k$, a graph $H$ with treewidth at most $k$, and connected subgraphs $H_1,\ldots, H_p$ of $H$.
A tree decomposition $\mathcal{T}$ of $H$ with width $k$ can be computed in linear time using, e.g., the algorithm of Bodlaender~\cite{MR1417901}.
Then, as shown by Bodlaender, Gustedt, and Telle in the proof of~\cite[Lemma~2.4]{MR1642971}, the tree decomposition $\mathcal{T}$ can be transformed efficiently into a tree decomposition $\mathcal{T}'$ of the intersection graph $G$ of $H_1,\ldots, H_p$ such that each bag of $\mathcal{T}'$ is the union of at most $k+1$ cliques in $G$.
In particular, $\mathcal{T}'$ has independence number at most $k+1$ and thus \cref{thm:bounded-tree-independence-number} applies, implying the result of Krause.

In the third paper of the series~\cite{dallard2022secondpaper}, we present two particular examples of where the complexity of solving the \textsc{Maximum Weight Independent Set} problem is significantly improved when using the result of \cref{thm:bounded-ell-refined-tree-independence-number} as compared to using \cref{thm:bounded-tree-independence-number}.

\section{Cameron--Hell construction}\label{sec:packing}

Given a graph $G$ and a family $\HH =\{H_j\}_{j\in J}$ of connected subgraphs of $G$, we denote by $G(\HH)$ the graph with vertex set $J$, in which two distinct elements $i,j\in J$ are adjacent if and only if $H_i$ and $H_j$ either have a vertex in common or there is an edge in $G$ connecting them.
This construction was considered by Cameron and Hell in~\cite{MR2190818}, who focused on the following particular case.
Given a graph $G$ and a (finite or infinite) set $\FF$ of connected graphs, we denote by $\HH(G,\FF)$ the family of all subgraphs of $G$ isomorphic to a member of $\FF$.
In particular:
\begin{itemize}
\item for $\HH = \HH(G,\{K_1\})$, the construction of the graph $G(\HH)$ is trivial, with $G(\HH)\cong G$, and
\item for $\HH = \HH(G,\{K_2\})$, the graph $G(\HH)$ is isomorphic to the square of the line graph of $G$.
\end{itemize}
A construction similar to $G(\HH)$ was studied by Duchet~\cite{MR778751}.
More recently, Gartland et al.~\cite{10.1145/3406325.3451034} considered the special case when $\HH$ consists of all connected induced subgraphs of $G$, referring in this case to the derived graph $G(\HH)$ as the \emph{blob graph} of $G$.

Cameron and Hell proved in~\cite{MR2190818} that for any chordal graph $G$, any set $\FF$ of connected graphs, and $\HH = \HH(G,\FF)$, the graph $G(\HH)$ is chordal, generalizing an analogous result due to Cameron in the case when $\FF = \{K_2\}$~\cite{MR1011265}.
Using \cref{chordal}, the result of Cameron and Hell states that $\tin(G(\HH))\le 1$ whenever $\tin(G)\le 1$.
We now generalize this result by showing that mapping any graph $G$ to the graph $G(\HH)$, where $\HH$ is an arbitrary collection of nonnull connected subgraphs of $G$, cannot increase the tree-independence number.
In particular, this shows that for any graph class $\mathcal{G}$ with bounded tree-independence number and any set $\FF$ of connected nonnull graphs, the class $\{G(\HH):G\in \G, \HH = \HH(G,\FF)\}$ also has bounded tree-independence number.

\begin{lemma}\label{tin-of-G(H)}
Let $G$ be a graph, let \hbox{$\mathcal{T} = (T, \{X_t\}_{t\in V(T)})$} be a tree decomposition of $G$, and let $\HH =\{H_j\}_{j\in J}$ be a finite family of connected nonnull subgraphs of $G$.
Then \hbox{$\mathcal{T}' = \big(T, \{\Bag'_t\}_{t\in V(T)}\big)$} with $\Bag'_t = \{j\in J : V(H_j) \cap X_t \neq \emptyset\}$ for all $t \in V(T)$ is a tree decomposition of $G(\HH)$ such that $\alpha(\mathcal{T}')\le \alpha(\mathcal{T})$.
\end{lemma}

\begin{proof}
Let us first show that $\mathcal{T}'$ is a tree decomposition of $G(\HH)$.
First, note that since $V(G(\HH))= J$, for each $t\in V(T)$ the set $X_t'$ is indeed a subset of $V(G(\HH))$.

Let $j\in J$ be a vertex of $G(\HH)$.
Fix a vertex $v\in V(H_j)$ and consider any bag $X_t$ of $\mathcal{T}$ such that $v\in X_t$. Then $v\in V(H_j)\cap X_t$ and hence $j\in X_t'$.
Thus, every vertex of $G(\HH)$ belongs to a bag of~$\mathcal{T}'$.

Let $\{i,j\}$ be an edge of $G(\HH)$.
Assume first that the subgraphs $H_i$ and $H_j$ have a vertex in common, say $v$.
Since $\mathcal{T}$ is a tree decomposition of $G$, there exists some $t\in V(T)$ such that $v\in X_t$.
The fact that $v\in V(H_i)\cap X_t$ implies that $i\in X_t'$, and similarly, $j\in X_t'$.
Assume now that there exist vertices $u\in V(H_i)$ and $v\in V(H_j)$ such that $uv\in E(G)$.
Since $\mathcal{T}$ is a tree decomposition of $G$, there exists some $t\in V(T)$ such that $\{u,v\}\subseteq X_t$.
The fact that $u\in V(H_i)\cap X_t$ implies that $i\in X_t'$.
Similarly, the fact that $v\in V(H_j)\cap X_t$ implies that $j\in X_t'$.
Hence, for every edge of $G(\HH)$ there exists a bag of $\mathcal{T}'$ containing both endpoints of the edge.

Next, consider an arbitrary vertex $j\in J$ of $G(\HH)$.
We need to show that the set of nodes $t\in V(T)$ such that $j\in X_t'$ induces a connected subgraph of $T$.
Let us denote for each vertex $v\in V(H_j)$ by $T_v$ the subgraph of $T$ induced by the nodes $t\in V(T)$ such that $v\in X_t$.
Since $\mathcal{T}$ is a tree decomposition of $G$, each $T_v$ is a connected subgraph of $T$, that is, a subtree.
For a node $t\in V(T)$, the condition $j\in X_t'$ is equivalent to the condition $V(H_j)\cap X_t\neq \emptyset$, that is, there exists a vertex $v\in V(H_j)$ such that $v\in X_t$.
Therefore, $j\in X_t'$ if and only if there exists a vertex $v\in V(H_j)$ such that $t$ belongs to the tree $T_v$.
It thus suffices to show that the union $T_j$ of the trees $T_v$ over all vertices $v\in V(H_j)$ forms a connected graph.
Suppose for a contradiction that $T_j$ is not connected and fix a component $C$ of $T_j$.
Let us denote by $U$ the set of vertices $u\in V(H_j)$ such that $V(T_u)\subseteq V(C)$.
Since $V(C)\neq V(T_j)$, we have $U\neq V(H_j)$.
By the connectedness of $H_j$, there is an edge $uv\in E(H_j)$ such that $u\in U$ and $v\in V(H_j)\setminus U$.
Let $t$ be a node of $T$ such that $\{u,v\}\subseteq X_t$.
The trees $T_u$ and $T_v$ both contain node $t$, and hence the connected component $C$ of $T_j$ contains both $T_u$ and $T_v$.
This implies that $v\in U$, a contradiction.

It remains to show that $\alpha(\mathcal{T})\ge \alpha(\mathcal{T}')$.
Let $t$ be a node of $T$ maximizing the independence number of the subgraph of $G(\HH)$ induced by $X_t'$.
Let $k$ be this independence number and let $I\subseteq J$ be an independent set of cardinality $k$ in the subgraph of $G(\HH)$ induced by $X_t'$.
Then for any two distinct elements $i,j\in I$ the graphs $H_i$ and $H_j$ are vertex-disjoint subgraphs in $\HH$ such that no edge of $G$ has one endpoint in $H_{i}$ and the other one in $H_{j}$.
Each element $i\in I$ belongs to the bag $X_t'$, which implies that $V(H_i)\cap X_t\neq \emptyset$, and hence there exists a vertex $u_i$ in $V(H_i)\cap X_t$.
Since any two distinct vertices $u_i$ and $u_j$ belong to the subgraphs $H_i$ and $H_j$ of $G$, which are vertex-disjoint and with no edges between them, the set $\{u_i:i\in I\}$ is an independent set of cardinality $k$ in the subgraph of $G$ induced by $X_t$.
This shows that $\alpha(\mathcal{T})\ge k = \alpha(\mathcal{T}')$, as claimed.
\end{proof}

\cref{tin-of-G(H)} immediately implies the following.

\begin{theorem}\label{tin-can-only-go-down-when-computing-G(HH)}
Let $G$ be a graph and let $\HH$ be a finite family of connected nonnull subgraphs of $G$.
Then $\tin(G(\HH))\le \tin(G)$.
\end{theorem}

Note in particular that \cref{tin-can-only-go-down-when-computing-G(HH)} implies that for any graph $G$, the tree-independence number of its blob graph is bounded by $\tin(G)$.

We will also make use of an algorithmic version of \cref{tin-of-G(H)}.
In the analysis of the time complexity we will make use of the following standard lemma (see, e.g., \cite{MR2063679} for a similar argument for sorting adjacency lists of graphs).

\begin{lemma}\label{lem:sorting}
Let $V = \{v_1,\ldots, v_n\}$ be a set and let $\mathcal{S} =\{S_i\}_{i\in I}$ be a finite family of subsets of $V$.
Then there exists an algorithm running in time $\mathcal{O}(|V|+|I|+\sum_{i\in I}|S_i|)$ that sorts each set $S_i$ with respect to the ordering $v_1,\ldots, v_n$ of $V$.
\end{lemma}

\begin{proof}
Let $B$ be the bipartite incidence graph of the family $\mathcal{S}$, that is, $B$ has vertex set $V\cup I$, and edge set $\{\{v,i\}: v\in V, i\in I, v\in S_i\}$.
We can compute the adjacency lists of the graph $B$ in time $\mathcal{O}(|V|+|I|+\sum_{i\in I}|S_i|)$ as follows.
We fix an ordering of the set $I$, say $I = \{i_1,\ldots, i_m\}$.
For each $i\in I$, the set $S_i$ already gives the adjacency list of $i$.
We initialize the adjacency lists for each $v\in V$ to the empty lists.
For all $j= 1,\ldots, m$, we iterate over the elements $v$ of $S_{i_j}$ and add $i_j$ to the end of the adjacency list of vertex $v$.
We now have the adjacency lists of all vertices of $B$; those for $v\in V$ are already sorted, while those for $i\in I$ need not be.

To sort the adjacency lists for vertices $i\in I$, we iterate over the adjacency lists of vertices $v\in V$ in a similar way as we did above for $i\in I$.
We reset the adjacency lists for all $i\in I$ to the empty lists.
For all $i = 1,\ldots, n$, we iterate over the elements $j\in I$ of the adjacency list of $v_i$ and add $v_i$ to the end of the adjacency list for vertex $j\in I$.
At the end of this procedure, the adjacency list of each $i\in I$ will contain exactly the elements of $S_i$, sorted with respect to the ordering $v_1,\ldots, v_n$ of $V$.
The total time complexity of the procedure is proportional to the number of vertices and edges of the graph $B$, that is, $\mathcal{O}(|V|+|I|+\sum_{i\in I}|S_i|)$.
\end{proof}

\begin{sloppypar}
\begin{corollary}\label{tin-of-G(H)-algorithmic}
There exists an algorithm that takes as input a graph $G$, a finite family $\HH =\{H_j\}_{j\in J}$ of connected nonnull subgraphs of $G$, and a tree decomposition \hbox{$\mathcal{T} = \big(T, \{\Bag_t\}_{t\in V(T)}\big)$} of $G$, and computes in time $\mathcal{O}(|J| \cdot ((|J| + |V(T)|) \cdot |V(G)| + |E(G)|))$ the graph $G(\HH)$ and a tree decomposition \hbox{$\mathcal{T}' = \big(T, \{\Bag'_t\}_{t\in V(T)}\big)$} of $G(\HH)$ such that $\alpha(\mathcal{T}')\le \alpha(\mathcal{T})$.
\end{corollary}
\end{sloppypar}

\begin{proof}
Fix an arbitrary ordering of the vertex set of $G$.
Using \cref{lem:sorting}, we first sort the vertex set of each of the graphs $H_j$, $j\in J$, as well as each bag $\Bag_t$, $t\in V(T)$, with respect to the fixed ordering of $V(G)$, in time \[\mathcal{O}\Bigg(|V(G)|+|J|+\sum_{j\in J}|V(H_j)|\Bigg)+ \mathcal{O}\Bigg(|V(G)|+|V(T)|+\sum_{t\in V(T)}|\Bag_t|\Bigg) = \mathcal{O}\Big((|J|+|V(T)|)\cdot|V(G)|\Big)\,.\]
Note that using this sortedness assumption, we can compute the union and the intersection of any two sorted sets $X,Y\subseteq V(G)$ in time $\mathcal{O}(|V(G)|)$.
To compute the graph $G(\HH)$, we only need to explain how to compute its edge set, since the vertex set is $J$.
For each $j\in J$, we perform a BFS traversal up to distance two from a new vertex $v_j$ added to $G$ which we make adjacent to all the vertices of $H_j$.
Let $R_j$ be the set of vertices of $G$ reached this way.
Then, for all $i\in J\setminus\{j\}$, the graph $H_i$ is adjacent to $H_j$ in $G(\HH)$ if and only if at least one vertex of $H_i$ belongs to $R_j$.
This can be tested in time $\mathcal{O}(|V(G)|)$ by first sorting the set $R_j$ and then computing the intersection $V(H_i)\cap R_j$.
Hence, this procedure over all $j \in J$ can be carried out in time $\mathcal{O}(|J|\cdot(|V(G)|+|E(G)|)) + \mathcal{O}(|J|^2\cdot|V(G)|) = \mathcal{O}(|J|^2\cdot |V(G)|+|J|\cdot|E(G)|)$.

To compute $\mathcal{T}'$, we need to compute for each $t\in V(T)$ the bag $\Bag_t'$ consisting of all vertices $j\in J$ of $G(\HH)$ such that $V(H_j) \cap X_t \neq \emptyset$.
All the intersections $V(H_j) \cap X_t$ can be computed in time
$\mathcal{O}(|J| \cdot |V(T)| \cdot |V(G)|)$.
Thus, the total time complexity of the algorithm is
$\mathcal{O}(|J|^2 \cdot |V(G)|+|J| \cdot |E(G)|+ |J|\cdot|V(T)|\cdot|V(G)|) = \mathcal{O}(|J| \cdot ((|J| + |V(T)|) \cdot |V(G)| + |E(G)|))$.
\end{proof}

\section{Weighted independent packings}\label{sec:implications}

Consider again a graph $G$ and a family $\HH =\{H_j\}_{j\in J}$ of connected subgraphs of $G$.
A subfamily $\HH'$ of $\HH$ is said to be an \emph{independent $\HH$-packing} in $G$ if every two graphs in $\HH'$ are vertex-disjoint and there is no edge between them, that is, $\HH'$ is an independent set in the graph $G(\HH)$.
Assume now that the subgraphs in $\HH$ are equipped with a weight function $w:J\to \mathbb{Q}_+$ assigning weight $w_j$ to each subgraph $H_j$.
For any set $I\subseteq J$, we define the \emph{weight} of the family $\HH' = \{H_i\}_{i\in I}$ as the sum $\sum_{i\in I}w_i$.
Given a graph $G$, a finite family $\HH = \{H_j\}_{j\in J}$ of connected nonnull subgraphs of $G$, and a weight function $w:J\to \mathbb{Q}_+$ on the subgraphs in $\HH$, the \textsc{Max Weight Independent Packing} problem asks to find an independent $\HH$-packing in $G$ of maximum weight.
In the special case when $\mathcal{F}$ is a fixed finite family of connected graphs and $\HH = \HH(G,\FF)$ is the set of all subgraphs of $G$ isomorphic to a member of $\FF$, we obtain the \textsc{Max Weight Independent $\mathcal{F}$-Packing} problem.
This is a common generalization of several problems studied in the literature, including:
\begin{itemize}
\item the \textsc{Independent $\mathcal{F}$-Packing} problem (see~\cite{MR2190818}), which corresponds to the unweighted case,
\item the \textsc{Max Weight Independent Set} problem, which corresponds to the case $\mathcal{F} = \{K_1\}$,
\item the \textsc{Max Weight Induced Matching} problem (see, e.g.,~\cite{MR3776983,MR4151749}), which corresponds to the case $\mathcal{F} = \{K_2\}$,
\item the \textsc{Dissociation Set} problem (see, e.g.,~\cite{MR3593941,MR615221,MR2812599}), which corresponds to the case when $\mathcal{F}= \{K_1,K_2\}$ and the weight function assigns to each subgraph $H_j$, $j\in J$, the weight equal to $|V(H_j)|$, and
\item the \textsc{$k$-Separator} problem (see, e.g.,~\cite{MR3987192,MR3296270}), which corresponds to the case when $\mathcal{F}$ contains all connected graphs with at most $k$ vertices, the graph $G$ is equipped with a vertex weight function $w:V(G)\to \mathbb{Q}_+$, and the weight function on $\HH$ assigns to each subgraph $H_j$, $j\in J$, the weight equal to $\sum_{x\in V(H_j)}w(x)$.
\end{itemize}

The \textsc{Max Weight Independent Packing} problem can be reduced in polynomial time to the \textsc{Max Weight Independent Set} problem, by means of the following.

\begin{observation}\label{reduction-to-MWIS}
Let $G$ be a graph, let $\HH = \{H_j\}_{j\in J}$ be a finite family of connected nonnull subgraphs of $G$, and a let $w:J\to \mathbb{Q}_+$ be a weight function on the subgraphs in $\HH$.
Let $I$ be an independent set in $G(\HH)$ of maximum weight with respect to the weight function $w$.
Then $I$ is an independent $\HH$-packing in $G$ of maximum weight.
\end{observation}

Using \cref{tin-of-G(H)-algorithmic,reduction-to-MWIS,thm:bounded-tree-independence-number}, we can now obtain an analogous result to \cref{thm:bounded-tree-independence-number} for the \textsc{Max Weight Independent Packing} problem.

\begin{sloppypar}
\begin{theorem}\label{thm:bounded-tree-independence-number-packings}
Let $k$ be a positive integer.
Then, given a graph $G$ and a finite family $\mathcal{H} = \{H_j\}_{j \in J}$ of connected nonnull subgraphs of $G$, the \textsc{Max Weight Independent Packing} problem can be solved in time \hbox{$\mathcal{O}(|J| \cdot ((|J| + |V(T)|) \cdot |V(G)| + |E(G)|+|J|^{k}\cdot|V(T)|)$} if $G$ is given together with a tree decomposition \hbox{$\mathcal{T} = (T, \{\Bag_t\}_{t\in V(T)})$} with independence number at most $k$.
\end{theorem}
\end{sloppypar}

\begin{proof}
Let $G$ be a graph, let $\HH = \{H_j\}_{j\in J}$ be a finite family of connected nonnull subgraphs of $G$, let $w:J\to \mathbb{Q}_+$ be a weight function on the subgraphs in $\HH$, and let \hbox{$\mathcal{T} = (T, \{\Bag_t\}_{t\in V(T)})$} be a tree decomposition of $G$ with independence number at most $k$.
By \cref{tin-of-G(H)-algorithmic}, we can compute in time $\mathcal{O}(|J| \cdot ((|J| + |V(T)|) \cdot |V(G)| + |E(G)|))$ the graph $G(\HH)$ and a tree decomposition \hbox{$\mathcal{T}' = \big(T, \{\Bag'_t\}_{t\in V(T)}\big)$} of $G(\HH)$ with independence number at most $k$.
Using \cref{thm:bounded-tree-independence-number}, we now compute in time $\mathcal{O}(|J|^{k+1}\cdot|V(T)|)$ an independent set $I$ in $G(\HH)$ of maximum weight with respect to the weight function $w$.
By \cref{reduction-to-MWIS}, $I$ is a maximum-weight independent $\HH$-packing in $G$.
The claimed running time follows.
\end{proof}

The case when the subgraphs in $\HH$ have bounded order is of particular interest, as it generalizes the \textsc{Max Weight Independent $\FF$-Packing} problem.
For this case, the time complexity can be slightly improved compared to what would be obtained by a direct application of \cref{thm:bounded-tree-independence-number-packings}.

\begin{sloppypar}
\begin{theorem}\label{prop:bounded-tree-independence-number-packings}
Let $k$ and $r$ be two positive integers.
Then, given a graph $G$ and a finite family $\mathcal{H} = \{H_j\}_{j \in J}$ of connected nonnull subgraphs of $G$ such that $|V(H_j)|\le r$ for all $j\in J$, the \textsc{Max Weight Independent Packing} problem can be solved in time \hbox{$\mathcal{O}\left(|V(G)|^{r(k+1)}\cdot|V(T)|\right)$} if $G$ is given together with a tree decomposition \hbox{$\mathcal{T} = (T, \{\Bag_t\}_{t\in V(T)})$} with independence number at most $k$.
\end{theorem}
\end{sloppypar}

\begin{sloppypar}
\begin{proof}
We assume that $G$ is represented with an adjacency matrix, since otherwise we can first compute such a representation from the adjacency lists in time $\mathcal{O}(|V(G)|^2)$.
Note that $|V(G(\HH))| = |J| = \mathcal{O}(|V(G)|^r)$, since by assumption each graph in $\HH$ has at most $r$ vertices, and for any such set of vertices we have at most $2^{r(r-1)/2} = \mathcal{O}(1)$ choices for the edge set.
We compute the edge set of $G(\HH)$ in time $\mathcal{O}(|V(G(\HH))|^2) = \mathcal{O}(|V(G)|^{2r})$, as follows.
For every two distinct $i,j\in J$, we check in time $\mathcal{O}(\max\{|V(H_i)|,|V(H_j)|\}) = \mathcal{O}(1)$ if $H_i$ and $H_j$ have a vertex in common.
If this is the case, then we add $\{i,j\}$ to the edge set of $G(\HH)$.
If this is not the case, then we check in time $\mathcal{O}(|V(H_i)|\cdot|V(H_j)|) = \mathcal{O}(1)$ if there is an edge in $G$ connecting a vertex of $H_i$ with a vertex of $H_j$.
If this is the case, then we add $\{i,j\}$ to the edge set of $G(\HH)$.

For the rest of the proof, we use the same approach as in the proof of \cref{thm:bounded-tree-independence-number-packings}.
In particular, we compute the tree decomposition $\mathcal{T}'$ of $G(\HH)$ in time $\mathcal{O}(|J| \cdot |V(T)|\cdot |V(G)|) = \mathcal{O}(|V(G)|^{r+1}\cdot |V(T)|)$, and a maximum-weight independent set $I$ in
$G(\HH)$ in time \hbox{$\mathcal{O}(|J|^{k+1}\cdot|V(T)|) = \mathcal{O}(|V(G)|^{r(k+1)}\cdot|V(T)|)$}.
The total time complexity of the algorithm is $\mathcal{O}(|V(G)|^{2r} + |V(G)|^{r+1}\cdot |V(T)| + |V(G)|^{r(k+1)}\cdot|V(T)|)$, which simplifies to $\mathcal{O}(|V(G)|^{r(k+1)}\cdot|V(T)|)$, as claimed.
\end{proof}
\end{sloppypar}

\begin{corollary}\label{cor:bounded-tree-independence-number-packings}
Let $\FF$ be a nonempty finite set of connected nonnull graphs and let $r$ be the maximum number of vertices of a graph in $\FF$.
Then, for every $k\ge 1$, the \textsc{Max Weight Independent $\FF$-Packing} problem is solvable in time $\mathcal{O}\left(|V(G)|^{r(k+1)}\cdot|V(T)|\right)$ if the input graph $G$ is given with a tree decomposition \hbox{$\mathcal{T} = (T, \{\Bag_t\}_{t\in V(T)})$} with independence number at most $k$.
\end{corollary}

\begin{remark}
The reader may wonder why we did not derive a result generalizing \cref{cor:bounded-tree-independence-number-packings} by using the notion of $\ell$-refined tree decompositions of bounded residual independence number, as we did in \cref{thm:bounded-ell-refined-tree-independence-number} for the \textsc{Max Weight Independent Set} problem.
The reason is that \cref{tin-of-G(H)} does not seem to generalize to $\ell$-refined tree decompositions, at least not so that the residual independence number would be preserved.

Consider the following example.
Fix two positive integers $\ell$ and $p$ and let $G$ be the tree consisting of a vertex $a$ adjacent to $\ell$ other vertices forming a set $B = \{b_1,\ldots, b_\ell\}$ such that each $b_i$ is also adjacent to $p$ vertices of degree one, forming a set $C_i$.
Let $\mathcal{T} = (T,\{X_t\}_{t\in V(T)})$ be a tree decomposition of $G$ such that $T$ is the graph $K_{1,p\ell}$, the high-degree node of $T$ is labeled with the bag $\{a\}\cup B$, and the $p\ell$ leaves of $T$ are labeled with bags corresponding to the edges in $G$ containing a vertex of degree one.
To make this tree decomposition $\ell$-refined, we set $U_t = B$ for the high-degree node $t$ of $T$ and $U_t = \emptyset$ for all the other nodes.
The residual independence number of this $\ell$-refined tree decomposition is $1$.
Let $\HH = \HH(G,\{K_2\})$ be the family of all connected subgraphs of $G$ of order two.
Then, the graph $G(\HH)$ is isomorphic to the graph obtained from the graph $K_{1,\ell}$ by substituting a clique of size $\ell+1$ into the vertex of degree $\ell$ and a clique of size $p$ into each vertex of degree one (see, e.g.,~\cite{MR3096332} for the definition of substitution).
Let \hbox{$\mathcal{T}' = (T, \{\Bag'_t\}_{t\in V(T)})$} be the tree decomposition of $G(\HH)$ as defined in  \cref{tin-of-G(H)}, that is, $\Bag'_t = \{H \in V(G(\HH)) : V(H) \cap \Bag_t \neq \emptyset\}$ for all $t\in V(T)$.
Then, there is no way to turn $\mathcal{T}'$ into an $f(\ell)$-refined tree decomposition of $G(\HH)$ with residual independence number $1$.
Indeed, consider the bag $X_t = \{a\}\cup B$ of $\mathcal{T}$ labeling the high-degree node of $T$.
Since every edge of $G$ has an endpoint in $X_t$, the bag of $\mathcal{T}'$ corresponding to node $t$  is $\Bag'_t = V(G(\HH))$.
Using the structure of the graph $G(\HH)$, we see that the smallest subset $U_t'$ of $X_t'$ such that the independence number of the subgraph of $G(\HH)$ induced by $X_t'\setminus U_t'$ is $1$ has size $(\ell-1)p$, which cannot be bounded from above by any function depending only on $\ell$.
\end{remark}

\section{Open questions (and some answers)}\label{sec:open-questions}

Our main motivation for the study of graph classes with bounded tree-independence number is the fact that for any finite family $\HH$ of connected graphs, the \textsc{Max Weight Independent Subgraph Packing} problem is solvable in polynomial time in any class of graphs with bounded and efficiently witnessed tree-independence number (cf.~\cref{thm:bounded-tree-independence-number-packings}).
This motivates the following question.

\begin{question}\label{computing-tree-decompositions}
Is there a computable function $f:\mathbb{Z}_+\to\mathbb{Z}_+$ such that for every positive integer $k$ there exists a polynomial-time algorithm that takes as input a graph $G$ with tree-independence number at most $k$ and computes a tree decomposition of $G$ with independence number at most $f(k)$?
\end{question}

A positive answer to \cref{computing-tree-decompositions} would imply a positive answer to the following question.

\begin{question}\label{mwis-complexity-for-bounded-tin}
Is the \textsc{Max Weight Independent Set} problem polynomial-time solvable in any class of graphs with bounded tree-independence number?
\end{question}

While we showed that the problem of computing the tree-independence number is \NP-hard, it remains an open question whether the problem can be solved in polynomial time for fixed values of this parameter.

\begin{question}\label{computing-constant-tree-independence-number}
For a fixed integer $k\ge 2$, what is the complexity of recognizing graphs with tree-independence number at most $k$?
\end{question}

For every graph $G$, it holds that $|V(G)| \leq \alpha(G) \cdot \chi(G)$.
Applying this inequality to every bag of a tree decomposition $\mathcal{T}$ of $G$ with $\alpha(\mathcal{T}) = \tin(G)$, we obtain that $\tw(G) +1 \leq \tin(G) \cdot \chi(G)$.
Seymour~\cite{MR3425243} introduced the \emph{tree-chromatic number} of a graph $G$, denoted $\treechi(G)$, as the smallest nonnegative integer $k$ such that $G$ admits a tree decomposition, each bag of which induces a $k$-colorable subgraph.
With a similar reasoning as above, we also obtain that $\tw(G)+1 \le \alpha(G)\cdot \treechi(G)$.
Do these two inequalities admit the following common strengthening?

\begin{question}\label{tree-alpha-and-tree-chi}
Does every graph $G$ satisfy $\tw(G)+1\le \tin(G)\cdot \treechi(G)$?
\end{question}

\Cref{tree-alpha-and-tree-chi} has a positive answer in any class of graphs in which the tree-chromatic number coincides with the chromatic number.
In particular, this is the case for any class of graphs in which the chromatic number coincides with the clique number (such as the class of perfect graphs), since in this case $\omega(G)\le \treechi(G)\le \chi(G)$, where the first inequality holds for any graph (see~\cite{MR3425243}).

\medskip
\noindent\textbf{Note.}
After the paper was submitted for publication, some of the above questions have been answered or partially answered.
In fact, \cref{computing-tree-decompositions} already admitted a positive answer due to the work of Yolov~\cite{MR3775804}, with $f(k) = \mathcal{O}(k^3)$, which we were not aware of at the time of submission.
The function $f$ (and the running time of the algorithm) has recently been improved to $f(k) = 8k$ by Dallard, Golovach, Fomin, Korhonen, and Milani{\v c}~\cite{dallard2022computing}.
Together with \cref{thm:bounded-tree-independence-number}, this implies that \cref{mwis-complexity-for-bounded-tin} also has a positive answer.
Dallard et al.~also showed in~\cite{dallard2022computing} that for every integer $k\ge 4$, the problem of recognizing graphs with tree-independence number at most $k$ is \NP-complete, leaving \cref{computing-constant-tree-independence-number} open only for $k\in \{2,3\}$.

\subsection*{Acknowledgements}

We are grateful to Erik Jan van Leeuwen for telling us about the work of Yolov~\cite{MR3775804}.
This work is supported in part by the Slovenian Research and Innovation Agency (I0-0035, research programs P1-0285 and P1-0383, research projects J1-3001, J1-3002, J1-3003, J1-4008, J1-4084, N1-0102, and N1-0160) and by the research program CogniCom (0013103) at the University of Primorska.

\bibliographystyle{abbrv}
\small
\bibliography{biblio}

\end{document}